\documentclass[12pt,reqno]{amsart}
\usepackage{amssymb,mathrsfs,color}
\usepackage{pinlabel}

\usepackage{amsmath, amsfonts, amsthm, verbatim}
\usepackage{epstopdf, mathtools}
\mathtoolsset{showonlyrefs}
\usepackage{hyperref}
\usepackage{color}

\makeatletter
\renewcommand{\l@section}{\@tocline{1}{0pt}{2pc}{}{}}
\renewcommand{\l@subsection}{\@tocline{2}{0pt}{3.5pc}{}{}}
\makeatother

\newtheorem{theorem}{Theorem}
\theoremstyle{plain}

\newtheorem{claim}{Claim}

\newtheorem{definition}{Definition}

\newtheorem{lemma}{Lemma}

\newtheorem{proposition}{Proposition}

\newtheorem{remark}{Remark}

\numberwithin{equation}{section}
\numberwithin{theorem}{section}
\numberwithin{proposition}{section}
\numberwithin{lemma}{section}
\numberwithin{cor}{section}
\numberwithin{definition}{section}
\numberwithin{remark}{section}
\numberwithin{claim}{section}

\newcommand{\ds}{ \displaystyle }
\newcommand{\p}{\partial}
\newcommand{\R}{\mathbb R}

\newcommand{\la}{\lambda}
\newcommand{\e}{\epsilon}

\newcommand{\E}{{\mathcal E}}

\let\oldbibliography\thebibliography
\renewcommand{\thebibliography}[1]{\oldbibliography{#1}
\setlength{\itemsep}{0pt}}

\begin{document}

\title[Corotational Harmonic Map Heat Flow]{Global solutions for the critical, higher-degree corotational harmonic map heat flow to $\mathbb{S}^2$}
\author{Stephen Gustafson}
\author{Dimitrios Roxanas}
\maketitle

\setcounter{page}{1}

%\unmarkedfntext{\textbf{Keywords:} Harmonic Map Heat Flow, Concentration Compactness, Global Existence, Stability. \textbf{2010 AMS Mathematics Classification:} 35K05, 35B40, 35B65}

\begin{abstract}
We study m-corotational solutions to the Harmonic Map Heat Flow from 
$\mathbb{R}^2$  to  $\mathbb{S}^2$.
We first  consider maps of zero topological degree, with initial energy below the threshold
given by twice the energy of the harmonic map solutions. For $m \geq 2$, we establish the
smooth global existence and decay of such solutions via the {\it concentration-compactness} 
approach of Kenig-Merle, recovering classical results of Struwe by this alternate
method. The proof relies on a profile decomposition, and the energy dissipation relation. 
We then consider maps of degree $m$ and initial energy {\it above} the harmonic map threshold energy, but below three times this energy. For $m \geq 4$, we establish the
smooth global existence of such solutions, and their decay to a harmonic map (stability),
extending results of Gustafson-Nakanishi-Tsai to higher energies. 
The proof rests on a stability-type argument used to rule out finite-time 
bubbling.
\end{abstract}

\tableofcontents

%------------------------------
\section{Introduction and Results}

The harmonic map heat flow into $\mathbb{S}^2$ is given by the equation
\begin{equation}
\textbf{u}_t = \Delta \textbf{u} + | \nabla \textbf{u} |^2 \textbf{u},  \hspace{1em} 
\textbf{u}(0,x) = \textbf{u}_0(x) \label{hmheat flow}
\end{equation}
where for $t \geq 0$,
\[
  \textbf{u}(t,\cdot) : \mathbb{R}^2 \rightarrow \mathbb{S}^2,
\]
\begin{equation*}\mathbb{S}^2 := \{ \textbf{u} = (u_1,u_2,u_3) : |\textbf{u}| = 1 \} \subset \mathbb{R}^3,\end{equation*}
is the unit 2-sphere, $\Delta$ denotes the Laplace operator in $\mathbb{R}^2$, 
and $\ds |\nabla \textbf{u}|^2 = \sum^2_{j=1} \sum^3_{i=1} (\frac{\partial u_i}{\partial x_j})^2$.
Equation ~\eqref{hmheat flow} is the $L^2$-gradient flow of the energy functional \begin{equation*} 
  \E(\textbf{u}) = \frac{1}{2} \int_{\mathbb{R}^2} |\nabla \textbf{u}|^2 dx
\end{equation*}
for such maps.
Taking formally the scalar product of the PDE with $\textbf{u}_t$ 
and integrating over $[0,t) \times \mathbb{R}^2$, we obtain
\begin{equation*} 
  \E(\textbf{u}(t,\cdot)) +  \int^t_0 \int_{ \mathbb{R}^2 } |\textbf{u}_t|^2 = \E(\textbf{u}_0) 
\end{equation*}
which implies that the energy is non-increasing. 
A more geometric way to write \eqref{hmheat flow} is
\begin{equation*}\textbf{u}_t = \sum^2_{j=1} D_j \partial_j \textbf{u} = P^{\textbf{u}} \Delta \textbf{u}, \end{equation*}
where $P^{\textbf{u}}$ denotes the orthogonal projection from $\mathbb{R}^3$ onto the tangent plane
\begin{equation*}
T_{\textbf{u}} \mathbb{S}^2 := \{ \boldsymbol \xi \in \mathbb{R}^3 : \boldsymbol \xi \cdot \textbf{u} = 0 \}
\end{equation*}
to $\mathbb{S}^2$ at $\textbf{u}$,
$\partial_j = \frac{\partial}{\partial x_j}$ is the usual partial derivative and $D_j$ the covariant derivative acting on vector fields $\xi(x) \in T_{\textbf{u}(x)} \mathbb{S}^2 :$
\begin{equation*}
D_j \xi := P^{\textbf{u}} \partial_j \xi = \partial_j \xi -(\partial_j \xi \cdot \textbf{u})\textbf{u} = \partial_j \xi + (\partial_j \textbf{u} \cdot \xi) \textbf{u}.
\end{equation*}
%In words, consider the connection on $\mathbb{S}^2$ induced (just by inclusion) from $\mathbb{R}^3$ (just the standard directional derivative)
%and to make an affine connection of it, project it down to the tangent plane of $\mathbb{S}^2,$ 

The harmonic map heat flow between Riemannian manifolds was introduced by 
Eells-Sampson \cite{ES} to study harmonic maps, which are its static solutions.  
Equation \eqref{hmheat flow}, where the target manifold is $\mathbb{S}^2$,
is, in addition, physically relevant as the purely diffusive case of the
{\it Landau-Lifshitz equations} of ferromagnetism \cite{KIK}.  
The setting of a two-dimensional domain is therefore of physical importance, 
but is also analytically interesting as the {\it energy-critical} one: the scaling 
$u(x) \mapsto u(\la x)$ leaves both the equation and the energy invariant
\begin{equation*} 
  \E(\textbf{u}(\cdot)) = \E(\textbf{u}(\frac{\cdot}{\lambda})),
\end{equation*}
and so this is the borderline case for 
smooth global existence versus possible singularity formation.
  
The question of singularity formation and characterization of possible blow-up has attracted a lot of attention. On a compact manifold domain, Struwe \cite{St} 
constructed a global weak solution whose singularities occur through energy concentration 
at a finite number of space-time points, at each of which a non-trivial harmonic map
bubbles off: for $t_n \nearrow T$,
\begin{equation*} 
  \textbf{u}(t_n, a_n + \lambda(t_n) x) \rightarrow \textbf{Q}(x), 
  \hspace{0.5em} \lambda(t_n) \rightarrow 0, \; a_n \to a,
  \quad \textbf{Q} \mbox{ harmonic}
\end{equation*}
locally in space.
Later work \cite{Q,DT,QT,Ta,Tb} (see also the book \cite{LW}) showed that
at a singularity, all the energy is accounted for by the bubbles and
the the weak limit ({\it body map}), and therefore that the solution converges strongly
to the body map, after all the bubbles are removed.

 Working in the subclass of the $m$-corotational solutions with $m = 1$, on a disk, 
 \cite{CDY} showed that, indeed, finite time blow-up does occur in some situations, 
using the sub-solution method. 
Formal analysis \cite{BHK}, and later rigorous constructions \cite{RSa,RSb},  
show that for 1-corotational maps described by the azimuthal angle $u(r,t)$,  
approaching a blow-up time $t \nearrow T$, 
\begin{equation*}
\begin{split} &u(t,r) - Q(\frac{r}{\lambda(t)}) \rightarrow u^* \;\;
\text{ in } \hspace{0.5em}  \dot{H}^1, \\ &\lambda(t) = c(u_0) (1+ o_{t\rightarrow T}(1)) \frac{(T-t)^L}{|\log(T-t)|^{\frac{2L}{2L-1}}}, c(u_0) > 0,\end{split} 
\end{equation*}
where $Q$ corresponds to the unique (up to scaling) harmonic map in this class,
and $L \in \mathbb{Z}^+$, with $L=1$ providing the generic blow-up rate. 
See also \cite{MDaW} for a related recent result,
and \cite{Bonn1, Bonn2,Ibrahim1, Ibrahim2} concerning the breakdown of solutions in higher (supercritical) dimensions.

On the other hand, Grotowski-Shatah \cite{GS}, using maximum principle methods, 
showed that on the unit disc in $\mathbb{R}^2$, 
$m$-corotational solutions will not blow-up in finite-time for degrees $ m \geq 2$, 
given certain pointwise bounds on the initial data.
One of our goals is to extend this result to the domain $\mathbb{R}^2$,
and, more importantly, to give a {\it maximum principle-free proof}, which one 
can therefore hope might extend to systems such as the Landau-Lifshitz equations. 

In this work we specialize to {\it $m$-co-rotational} maps: in polar coordinates,
\begin{equation*} 
  \textbf{u}(t, (r,\theta)) = (\cos (m \theta) \sin(u (t, r)), \sin (m \theta) \sin(u (t, r)), \cos(u (t, r))).
\end{equation*}
for which~\eqref{hmheat flow} reduces to the problem
\begin{equation} 
  u_t = u_{rr} + \frac{1}{r} u_r - m^2 \frac{\sin2u}{2r^2},
  \qquad u(0,r) = u_0(r) 
\label{hmhf} 
\end{equation}
for the angle $u(t,r)$. Without loss of generality we assume $m > 0.$
%Moreover, if (\ref{hmhf}) holds, then the other components of $\vec{u}$ are also easily seen to satisfy the heat-flow equation.
Defining
\[
\begin{split}
&\Delta_r u= u_{rr} + \frac{1}{r}u_r, \;\; \text{ the radial Laplacian in \hspace{0.00001em}} \mathbb{R}^2,  \\ 
&\Delta_m u = (\Delta_r - \frac{m^2}{r^2})u, \hspace{0.5em}  
  \end{split}
\]
we may write~\eqref{hmhf} as
\begin{equation*}
\begin{split}
  u_t &= (\Delta_r - \frac{m^2}{r^2}) u  +\frac{m^2}{r^2} (u- \frac{\sin2u}{2}) \\
  &= \Delta_m u  + F(u), \qquad F(u)= \frac{m^2}{r^2}(u- \frac{\sin2u}{2}).
\end{split}
\end{equation*}
The energy for these maps is given by 
\[
  \E(\textbf{u}) = 2 \pi E(u), \quad
  E(u) := \frac{1}{2} \int_0^{\infty} (u^2_r + m^2 \frac{\sin^2(u)}{r^2}) rdr.
 \]
Note that finite energy requires
\[
  \lim_{r \to 0, \; \infty} u(r) \in \pi \mathbb{Z},
\]
and indeed the assumption of finite energy is sufficient to guarantee the existence of 
these above limits (e.g.,\cite{GKTa}).
For $m$-corotational maps, the classical energy lower-bound 
by the topological degree reads 
\begin{equation}  \label{TLB}
\begin{split}
  E(u) &= \frac{1}{2} \int_0^\infty \left( u_r \pm \frac{m}{r} \sin(u) \right)^2 r dr
  \pm m \int_0^\infty \left( \cos(u) \right)_r d r \\
  & \geq \frac{1}{2} \int_0^\infty \left( u_r \pm \frac{m}{r} \sin(u) \right)^2 r dr
  + m | \cos(u(\infty)) - \cos(u(0)) | \\
  & \geq 2 \; | \text{degree}(\textbf{u}) |
\end{split}
\end{equation}
for the appropriate choice of $\pm$ sign.
The $m$-corotational stationary solutions -- corresponding to the harmonic maps -- are
the functions saturating this inequality, given by
\begin{equation} \label{Q}
  Q(r) = \pi - 2 \arctan(r^m), \quad Q_r + \frac{m}{r} \sin(Q) = 0,
  \;\; Q(0) = \pi, \; Q(\infty) = 0
\end{equation}
and their scalings $Q(\frac{r}{s})$, $s > 0$, as well as the negatives and 
shifts by $\pi \mathbb{Z}$ of these.
Since these harmonic maps each minimize the energy within their topological class, 
they provide natural thresholds for global
smoothness and  decay vs. singularity formation.

In light of the above considerations, we make the following definitions: 
 \begin{equation*}\begin{split}
&E_0 := \{ u: [0,\infty) \to \R \; | \; E(u) < 2 E(Q), \;\; 
\lim_{r \to 0+} u(r)= 0, \; \lim_{r \rightarrow \infty} u (r) = 0 \},\\
&E_1 := \{ u: [0,\infty) \to \R \; | \; E(Q) \leq  E(u) \leq 3 E(Q), 
\lim_{r \to 0+} u(r)=\pi,
\lim_{r \rightarrow \infty} u (r) = 0 \}.
\end{split}
\end{equation*}
and note that 
\[
  u \in E_0 \; \implies \; degree(\textbf{u}) = 0, \quad
  \min_{u \in E_0} E(u) = E(0) = 0, 
\] 
\[
  \quad \; u \in E_1 \; \implies \; degree(\textbf{u}) = m, \quad
  \min_{u \in E_1} E(u) = E(Q) =  2m.
\] 

Our first result concerns solutions in the ``below-threshold" class $E_0$:
\begin{theorem} \label{hmhf below threshold theorem} 
Assuming $u_0 \in E_0,$ and $m \geq 2$, \eqref{hmhf} has a unique solution 
$u(t,r)$, which is global in time, smooth, and decays:
$E(u(t,\cdot)) \to 0$ and $\ds \sup_r |u(t,r)| \to 0,$ as $t \to \infty$.
\end{theorem}
The main purpose here is to give a proof which follows Kenig-Merle's 
{\it concentration-compactness} strategy \cite{KMa}, originally developed for 
(and widely applied to) {\it dispersive} problems,
but relevant also to certain {\it diffusive} ones \cite{GKP1,GKP2,KK,GR}. 
The method is well-suited to the non-compact domain, and, more pertinently, provides an alternative approach to the classical theory of Struwe and successors. 

We first establish a local well-posedness theory for solutions of~\eqref{hmhf} in $E_0$
which parallels that for (say) the energy-critical nonlinear Schr\"o-dinger equation, and differs from that appearing in the classical parabolic literature.
Then the key tools for the concentration-compactness strategy are a stability-under-small-perturbations variant of the local theory, and a profile decomposition for an $\dot{H}^1$-like space, adapted to the heat flow.
A profile decomposition directly applicable to our setting was not readily available, stemming from the absence of some Sobolev embeddings in dimension two. 
So we take an indirect approach, first establishing estimates on the linear evolution 
in higher dimensions, which then connect back to our problem through a 
change of variable.

For applications of the concentration-compactness approach to other (below threshold) geometric problems we refer, for instance, to \cite{CKLSa} in the context of Wave Maps, and to \cite{BIKTa,BIKTb,GK} for Schr\"odinger Maps. A more comprehensive review of the literature can be found in \cite{thesis}.

To put our results in the ``above-threshold" class $E_1$ into context, we first recall 
results from the series of papers \cite{GKTa, GKTb,GGT,GNT} which apply
to the $m$-corotational heat-flow~\eqref{hmhf}, but more generally to
solutions of the {\it Landau-Lifshitz} family of equations
\begin{equation} 
\textbf{u}_t = a \left(\Delta \textbf{u} +  |\nabla \textbf{u}|^2\textbf{u} \right) 
+ b \; \textbf{u} \times \Delta \textbf{u}, \quad 
\textbf{u}(0,x) = \textbf{u}_0(x), \quad
a\geq 0, \; b \in \mathbb{R}
\label{LL} 
\end{equation}
of degree $m$, with {\it equivariant} symmetry.
For higher degrees, the $m$-equivariant harmonic maps are shown to be asymptotically
stable in the strong sense that if the initial data has near-minimal (harmonic) 
energy given the degree,
the solution is globally smooth and asymptotically converges to a nearby harmonic map: 
\begin{theorem}\label{regularity for LL}(\cite{GNT}).
Assume $\textbf{u}_0$ is of degree $m \geq 3$ with equivariant symmetry,  and 
\begin{equation*} 
  \mathcal{E}(\textbf{u}_0) - 4\pi m \ll 1.
\end{equation*} 
Then the solution of~\eqref{LL} is globally regular (continuous into the energy space)
and there exists a harmonic map $\textbf{Q}$ close to $\textbf{u}_0$
(in the energy norm) such that
\[
  \| \textbf{u}(t,\cdot) - \textbf{Q} \|_{L^{\infty}} + 
  a \mathcal{E}(\textbf{u}(t,\cdot) - \textbf{Q}) \rightarrow 0, 
  \hspace{0.3em} \text{as}  \hspace{0.3em} t \rightarrow \infty.
\]
\end{theorem}
\noindent
{\it Remarks:}
\begin{itemize}
\item
in the dissipative case ($a > 0$) these solutions are
converging to a harmonic map in the energy norm, while this is impossible for
the conservative case $a=0$, known as the {\it Schr\"odinger flow};
\item
the case $m=3$ is significantly more complex that $m \geq 4$, in particular
requiring a normal form-type argument (\cite{GNT}) to establish the
asymptotic behaviour -- for this reason we consider only $m \geq 4$ here;     
\item
for the $m = 2$ corotational heat-flow, the above conclusion is false: solutions 
are still global, but may exhibit blow-up in {\it infinite time},
or other complex behaviours (\cite{GNT});
\item
for $m=1$, near-minimal energy solutions may exhibit {\it finite-time} blow-up 
(\cite{RSa,RSb}).
\end{itemize}

Our main result is to extend this theorem, for the corotational heat-flow, beyond the perturbative regime to the higher energy maps in $E_1$: 
\begin{theorem} \label{above threshold theorem}
Assuming $u_0 \in E_1$ and $m \geq 4$, \eqref{hmhf} has a unique solution,
which is global in time, smooth, and converges to a harmonic map: 
for some $s > 0$, $\E(u(t,\cdot) - Q(r/s)) \to 0$ and
$\ds \sup_r |u(t,r) - Q(r/s)| \to 0,$ as $t \rightarrow \infty$. 
\end{theorem}
We would like to emphasize here that solutions in this class {\it are not prohibited from 
forming a singularity by either energetic or topological constraints} -- that these 
solutions remain globally smooth {\it does not follow from any classical theory}.     

The main point is to exclude the possibility of finite-time blowup.
As in \cite{St}, if the solution blows-up in finite time, it does so by bubbling off a non-trivial harmonic map. The corotational symmetry (and finite energy) ensures the only
possible concentration points are $r=0$ and $r=\infty$.
Finite-time energy concentration at spatial infinity is ruled out using the
energy dissipation relation.
The condition $E(u_0) \leq 3 E(Q)$ ensures only one bubble may form, and so 
following \cite{Q}, if $u(t,\cdot) \in E_1$ is a solution blowing up at time 
$T$, there exists a sequence of times $t_n \nearrow T$, scales 
$s_n \searrow 0$, and a function $w_0 \in E_0$ such that 
$\xi(t_n,\cdot) := u(t_n, \cdot) - Q(\frac{\cdot}{s_n}) - w_0 \to 0$ in the energy norm.
This is contradicted by adapting the modulation theory and linearized 
(about $Q$) evolution estimates of \cite{GNT} to estimate $\xi(t_n,\cdot)$,
and show $s_n \not\to 0$. We exploit the fact that in certain space-time
norms, the nonlinear interaction of $Q(\frac{r}{s_n})$ and 
the (smooth, global) solution emanating from data $w_0$ is
small on small time intervals.

The remaining impediment is possible infinite-time concentration. 
But in that scenario, the energy must approach the minimal energy $E(Q)$,
and so it is excluded by~\cite{GNT}. 

We emphasize that {\it the proof does not in any way rely on the maximum 
principle}. Therefore, the result may be extended beyond the corotational class
to the (larger) {\it equivariant} class, and even to the Landau-Lifshitz 
equations~\eqref{LL} (work in progress).

%-------------------------------------------------------------------
\section{Heat-Flow Below Threshold} \label{HMHF Below}

In this section we prove Theorem~\ref{hmhf below threshold theorem}
on the ``below-threshold" solutions of the corotational heat-flow.

 \subsection{Analytical ingredients}

\subsubsection{Energy properties of maps in $E_0$}

We begin by showing the function class $E_0$ is naturally endowed with the 
energy-space norm
$$
  \| u \|^2_{X^2} = \int_{0}^{\infty} \left(u_r^2 + m^2 \frac{u^2}{r^2} \right) r dr,
$$
in the following sense: given $\delta_1 > 0$, there is $C = C(\delta_1) > 0$ such that
\[
  u \in E_0, \; E(u) \leq 2 E(Q) - \delta_1 \; \implies \;
  \frac{1}{C} \| u \|_{X^2}^2 \leq E(u) \leq C \| u \|_{X^2}^2.
\]
This follows directly from a version of the topological lower bound~\eqref{TLB} 
localized to intervals:
\begin{lemma}
If $u \in E_0,$ with $E(u) \leq 2 E(Q)-\delta_1,$ for some $
\delta_1 > 0,$ there is $\delta_2 = \delta_2(\delta_1) > 0$ such that  
\begin{equation} 
  | u(r) | \leq \pi - \delta_2.
\end{equation}
\end{lemma}
\begin{proof}
As in \cite{SS,CKLSa} we define
\begin{equation*}
G(u) := \int_0^u m | \sin(s) | ds 
\end{equation*} 
and 
\begin{equation*}
E^{r_2}_{r_1}(u) := \frac{1}{2} \int_{r_1}^{r_2} (u^2_r + \frac{m^2 \sin^2(u)}{r^2}) rdr.
\end{equation*}
Then for all $0 \leq r_1 < r_2 < \infty$, by the Fundamental Theorem of Calculus and Young's inequality:
\begin{equation} \begin{split} 
| G(u(r_2)) - G(u(r_1)) | &= \left | \int_{r_1}^{r_2} \frac{\partial}{\partial r} G(u(r)) dr \right | 
=  \left| \int_{r_1}^{r_2} m | \sin u | u_r dr \right |\\ &\leq \frac{1}{2} \int_{r_1}^{r_2}
(\frac{m^2 \sin^2(u)}{r^2}+ u^2_r) rdr \leq E^{r_2}_{r_1}(u)  
\label{Gbnd} 
\end{split}\end{equation}
For $u \in E_0$, $G(u(\infty)) = G(0) = 0$ and $G(u(0)) = G(\pi) = 0$. 
From \eqref{Gbnd} for any $r>0$: 
\begin{equation*} 
| G(u(r))| = | G(u(r)) - G(u(0)) | \leq E^{r}_{0}(u)
\end{equation*} 
and
\begin{equation*} 
| G(u(r))| = | G(u(\infty)) - G(u(r)) | \leq E^{\infty}_{r}(u).
\end{equation*} 
Thus
\begin{equation*}
  | G(u(r)) | \leq \frac{1}{2} E(u) \leq \frac{1}{2} (2E(Q) - \delta_1) = 2m - \frac{\delta_1}{2}.
\end{equation*} 
$G$ is odd, increasing on $[-\pi, \pi]$, and $G(\pi) = 2m$, so
\[
  |u(r)| \leq G^{-1}(2m - \frac{\delta_1}{2}) =: \pi - \delta_2, \quad \delta_2 > 0. 
\]
\end{proof}
We remark that due to the boundary conditions, $E_0$ contains no nontrivial static solutions (corresponding to harmonic maps) since these are all monotone~\eqref{Q}.

\subsubsection{Local well-posedness for maps in $E_0$}

From now on, unless otherwise specified, all norms will be for functions defined
on $(0,\infty)$, with the measure $r dr$ 
We define spaces $r L^p$ via norm
\[
  \| u \|_{r L^p} := \left\| \frac{u}{r} \right\|_{L^p},
\] 
and $X^p$ via norms 
\[ 
  \| u \|^p_{X^p} := \| u_r \|_{L^p}^p + m^p  \| u \|_{r L^p}^p,
  \qquad \| u \|_{X^\infty} = \| u_r \|_{L^\infty} + \| u \|_{r L^\infty}. 
\]

Recall we may write the $m$-corotational heat-flow~\eqref{hmhf} 
for initial data $u_0 \in E_0$ as
\begin{equation} \label{corotationalCP}
 \left\{ \begin{array}{l} 
  u_t = \Delta_m u + F(u) \\
  u(0) = u_0 \in X^2
\end{array} \right. 
\end{equation}
where $\Delta_m = \Delta_r - \frac{m^2}{r^2}$ and
$F(u)= \frac{m^2}{r^2} (u- \frac{\sin2u}{2})$.
We will say that a function 
$u: I \times \mathbb{R} \rightarrow \mathbb{R}, \; I= [0,T)$,
is a {\it solution} to~\eqref{corotationalCP} on $I$ 
if $u \in  C_t X^2_r \cap L^4_t rL^4_r(K)$ for every compact $K \subset I$, 
and for every $t \in I$
\begin{equation}
  u(t) =  e^{t\Delta_m}u_0 + \int_0^t e^{(t-s)\Delta_m} F(u(s)) ds. 
\label{Duhamel}
\end{equation}
We summarize the local theory:
\begin{theorem} (Local well-posedness) \label{hmhfLWP}
\begin{enumerate} 
\item (Local Existence) 
Let $u_0 \in X^2$ There exists an $\epsilon > 0$ such that if 
$I = [0,T)$ and $ \| e^{t\Delta_m} u_0 \|_{L^4_t(I;rL^4)} < \epsilon,$ 
then there exists a unique solution to (\ref{corotationalCP}),
which moreover satisfies $\|u\|_{L^4_t(I;rL^4)} \leq 2 \epsilon$. 
To each initial datum $u_0$ we can associate a maximal time interval
$I = [0, T_{\text{max}}(u_0))$ on which there is a solution
($T_{\text{max}}(u_0)$ may be $ + \infty$).
\item (Blow-up Criterion) 
$T_{\text{max}}(u_0) < + \infty \; \implies
\| u \|_{L^4_t([0,T_{\text{max}}(u_0));  rL^4_r)} = + \infty$. 
\item (Energy dissipation) 
$u_t \in L^2_t([0,T_{max}(u_0)); L^2)$ 
and for each $t < T_{max}(u_0)$,
\[
  E(u(t)) + \int_0^t \int_0^\infty u_t^2(s) r dr \; ds = E(u_0).
\]
\item (Decay) 
If $T_{\text{max}}(u_0) = + \infty $ and $ \| u \|_{L^4_t([0,\infty);  rL^4_r)} < + \infty,$ then \\ 
$ \| u(t) \|_{X^2} \rightarrow 0$ as $ t \rightarrow + \infty.$
\item (Small data) 
If $ \|u_0\|_{X^2}$ is sufficiently small, then
$T_{max}(u_0) = \infty$ and  
$ \| u \|_{L^4_t([0,\infty);  rL^4_r)} \lesssim \| u_0 \|_{X^2}$; in particular
$\| u(t) \|_{X^2} \to 0$.
\item (Continuous Dependence) 
$T_{\text{max}}$ is a lower semi-continuous function of 
$u_0 \in X^2$, and $u_0 \to u(t,\cdot)$ is continuous on $X_2$.
\end{enumerate}
\end{theorem}
This local well-posedness rests on space-time estimates 
for the linear evolution $e^{t \Delta_m}$.
The decay estimate
\begin{equation} 
  \| e^{t \Delta_m} \phi \|_{L^p} \lesssim t^{-(1/a - 1/p)} \| \phi \|_{L^a}, 
  \quad 1 \leq a \leq p \leq \infty \label{STdecay} 
\end{equation}
is an immediate consequence of Young's inequality and the 
explicit heat kernel.

For $\phi = \phi(r)$, $f = f(t,r)$, the space-time estimates
\begin{equation} 
\label{STLpLq}
m \geq 2: \quad
\| e^{t \Delta_m} \phi \|_{S} \lesssim \| \phi\|_{X^2} \qquad
\| \int_0^t e^{(t-s)\Delta_m} f(s) ds  \|_{S} \lesssim  
\| f \|_{L^1_t X^2 + L^2_t X^1} 
\end{equation}
where
\[
  \| u \|_S := \| u \|_{L^\infty_t X^2} + \| u \|_{L^2_t X^\infty}
  + \| u \|_{L^2_t \; r^2 L^2} + \| u_r \|_{L^2_t X^2}
\]
follow from the standard heat equation energy estimate 
for the derived function $\nabla \left( e^{i \theta} u(t,r) \right)$,
and an interpolation estimate (\cite{GGT}). 
We remark that this procedure only yields the {\it endpoint}
spaces $u \in L^2_t X^\infty \cap L^2_t r^2 L^2$, $u_r \in L^2_t X^2$
under the restriction $m \geq 2$. For $m=1$, the estimate still holds if these 
spaces are replaced by some  
$L^r_t X^p$ with $\frac{1}{r} + \frac{1}{p} = \frac{1}{2}$, 
$p < \infty$, and this suffices for the local well-posedness.
However, for $m=1$, the endpoint space would be required below in the profile 
decomposition argument. That is the reason we impose
$m \geq 2$ here, and in Theorem~\ref{hmhf below threshold theorem}.

Given the space-time estimates~\eqref{STLpLq}, 
together will the elementary pointwise inequalities
\begin{equation}  \label{Fest}
  | \frac{1}{r} F(u) | \lesssim \frac{1}{r^3} |u|^3, \quad
  | \p_r F(u) | \lesssim \frac{1}{r^3} |u|^3 + \frac{1}{r^2} u^2 |u_r|
\end{equation}
on the nonlinearity, the proof of the local well-posedness is a
standard variant of the corresponding proof for the critical NLS, based 
on the Banach fixed-point theorem (see \cite{CW}). So we will omit most of the details (which can be found in \cite{thesis}),
and just indicate how to establish decay of global solutions
with finite space-time norm.

{\it Proof of the decay of global solutions:}
assuming  $\|u\|_{L^4 rL^4([0,\infty)} < \infty$, 
we first show that also $\|u\|_{S([0,\infty))} < \infty$. 
For a given $\tilde{\epsilon}$ (to be chosen small), 
subdivide the interval $[0, \infty)$ into a finite number of intervals 
$I_j = [t_j,t_{j+1})$ so that $\|u\|_{L^4 rL^4(I_j)} \leq \tilde{\epsilon}$. 
For $t \in I_j$, by the Duhamel formula,
\begin{equation*} 
  u(t) = e^{(t-t_j) \Delta_m} u(t_j) + \int_{t_j}^t e^{(t-s)\Delta_m} F(u(s)) ds.
\end{equation*} 
Using~\eqref{STLpLq},~\eqref{Fest}, and H\"older's inequality, we arrive at
\begin{equation*}
  \|u\|_{S(I_j)} \leq C \|u(t_j)\|_{X^2} + C \|u\|_{S(I_j)} \|u\|^2_{L^4 rL^4(I_j)}
  \leq C \|u(t_j)\|_{X^2} + C \tilde{\e}^2 \|u\|_{S(I_j)}.
\end{equation*}
Choosing $\tilde{\epsilon}$ such that $ C \tilde{\epsilon}^2 \leq \frac{1}{2}$ yields
\[
  \|u\|_{S(I_j)} \leq 2 C \|u(t_j)\|_{X^2}.
\]
Since there are only finitely many $I_j$, it follows that
$\| u \|_{S([0,\infty))} < \infty$.
So in particular, for any $\epsilon > 0$, there is a $T > 0$ such that
\[
  \|u\|_{S([T,+\infty))}^3 \leq \epsilon.
\] 
By the Duhamel formula, for $t \geq T$,
\begin{equation*}
  u(t) = e^{(t-T) \Delta_m} u(T) + \int_T^t  e^{(t-s)\Delta_m} F(u) ds.
\end{equation*}
By~\eqref{STLpLq} and~\eqref{Fest} as above,
\begin{equation*}
  \left\|  \int_T^t  e^{(t-s)\Delta_m} F(u) ds \right\|_{S([T,\infty))} \lesssim 
  \| F(u) \|_{L^{4/3}([T,\infty); X^{4/3})} \lesssim 
  \| u \|_{S([T,\infty))}^3 \lesssim \e. 
\end{equation*}
As well,
\[
  \| e^{(t-T)\Delta_m} u(T) \|_{X^2} \to 0 \mbox{ as } t \to \infty,
\]  
by~\eqref{STdecay} and the density of $X^1 \cap X^2$ in $X^2$. 
Therefore
\[
  \limsup_{t \to \infty} \| u \|_{X^2} \lesssim \e.
\]
Since $\e$ was arbitrary, the result follows.
$\Box$
\begin{remark}
Uniform decay follows from energy space decay by
the elementary embedding $X^2 \subset L^\infty$:
\[
  \| u(t) \|_{X^2} \to 0 \; \implies \; \| u(t) \|_{L^{\infty}} \lesssim \| u(t) \|_{X^2} \to 0.
\]
\end{remark}

Notice that our local theory combined with the previous section on the equivalence of the $X^2$ and the energy topology implies that if $u_0 \in E_0,$ the boundary conditions persist in time, i.e. $u(t,\cdot) \in E_0$ throughout its lifespan.

\subsubsection{Stability under perturbations}

An important extension of the local existence theory is the following 
``perturbation" or ``stability" theorem. This type of result, 
which establishes the existence of a solution to (\ref{corotationalCP}) nearby a 
given approximate one, goes back to \cite{Iteam, TV}, and is by now standard.
We use the following version (for a proof see \cite{thesis}):
\begin{theorem} \label{LT perturbations}(Stability)
Let $ I =[0,T)$ and let $\tilde{u}$ be defined on $I \times [0,\infty)$ with
\[ 
  \|\tilde{u}\|_{L^{\infty}(I;X^2)} \leq M, \qquad
  \|\tilde{u}\|_{L^4(I; rL^4)} \leq L
\]
for some $M, L > 0$, and set
\begin{equation*} 
  e := \tilde{u}_t - \Delta_m\tilde{u} - F(\tilde{u})
\end{equation*}
Let $u_0 \in X$ be such that 
\begin{equation}  
  \| u_0 - \tilde{u}(\cdot,0) \|_{X^2} \leq M' \label{ID} 
\end{equation} 
for some $M' > 0$.
There is $\epsilon_1 = \epsilon_1(M,M',L) > 0$ such that if for some 
$0 < \epsilon \leq \e_1$ the smallness conditions
\begin{eqnarray}
 \|e^{t \Delta_m}(u_0 - \tilde{u}(0,\cdot) \|_{L^4(I; rL^4)} \leq \epsilon \label{exp}
\\ \| e \|_{L^{4/3}(I; rL^{4/3})} \leq \epsilon 
\end{eqnarray}
hold, then there exists a solution of (\ref{corotationalCP}) with $u(0) = u_0$ 
satisfying
\begin{equation}
  \| u - \tilde{u} \|_{L^4 rL^4(I)} \leq C(M',M,L) M'. \label{eq3}
\end{equation}
\end{theorem}

\subsection{Profile decomposition}

The following proposition is the main tool (together with the stability theorem above)
in the concentration-compactness approach to establishing global existence and decay. 
The idea of a profile decomposition is to characterize the loss of compactness in some embedding, and to recover some compactness. 
It can be traced back to \cite{L,BC,St1,SU} and their modern ``evolution'' counterparts \cite{BG,KMa,KMb}.

%   In what follows, all the functions in $X^2$ are radial. \\
 \begin{proposition} \label{hmhfPD}
Let $\{u_n\}_n$ be a bounded sequence of radial functions in $X^2$. Then, after possibly passing to a subsequence (in which case, we rename it $u_n$ again), there exist a family of radial functions $ \{ \phi^j \}^{\infty}_{j=1} \subset X^2$ and scales $\lambda^j_n > 0$ such that
for each $J \geq 1$, 
\begin{equation} u_n (x) = \sum_{j=1}^J  \phi^j( \frac{x}{\lambda^j_n}) +  w^J_n(x), \end{equation}
where $w^J_n \in X^2$ is such that
\begin{equation} 
\lim_{J \rightarrow \infty} \limsup_{n \to \infty} \| e^{t\Delta_m}  w^J_n \|_{L^4_t rL^4_r} = 0, \label{holygrail} 
\end{equation}
\begin{equation}  
w^J_n(\lambda^j_n x) \rightharpoonup 0 \hspace{0.6em} \text{in} \hspace{0.6em} X^2, \hspace{0.6em}  \forall j \leq J. 
\label{wwconv hmhf} 
\end{equation}
Moreover, the scales are asymptotically orthogonal in the sense that 
\begin{equation} 
\frac{\lambda^j_n}{\lambda^{i}_n} + \frac{\lambda^{i}_n}{\lambda^j_n} \rightarrow + \infty, \hspace{0.6em} \forall i \neq j .
\label{scale orth2}
\end{equation}
Furthermore, $ \text{for all} \hspace{0.6em} J \geq 1$ we have the following decoupling properties:
\begin{equation}
\|u_n\|^2_{X^2} = \sum_{j=1}^J \| \phi^j \|^2_{X^2} + \| w^J_n \|^2_{X^2} + o_n(1) \label{H1dec} \end{equation}
\begin{equation} E(u_n) = \sum_{j=1}^J  E(\phi^j) + E (w^J_n) + o_n(1). \label{ED} \end{equation}
\end{proposition}
The procedure through which one establishes such a decomposition has become standard by now (e.g. see \cite{BG,K}), thus we will only present the equation-specific parts of the argument.

There are two general roadmaps to follow in establishing such a decomposition. To get the convergence of the error $w^J_n$ in the appropriate space-time norm, one can either use 
a refinement of the space-time estimates on the linear propagator, 
or a refinement of a Sobolev inequality through which the refinement of the space-time estimates will follow via interpolation arguments. 
The first approach would require more work in our case: arguments used in the 
Schr\"{o}dinger case cannot be applied directly 
due to the lack of an analogue of the restriction theorems used.
For the second approach, dimension two is very special due to the lack of the 
usual embeddings.

Our strategy is to first establish (\ref{holygrail})
for the homogeneous linear heat equation for radial functions in higher dimensions. We make use of a refined Sobolev inequality, first proved in \cite{BG} for $d=3$ and later generalized to 
$d > 3$ in \cite{Bu}. Then we convert this estimate to our 2d spaces by a change of 
variable, and use interpolation again to obtain the desired convergence.

\begin{definition}
An exponent pair $(q,p)$ is $L^2$-\textit{admissible} in dimension $d$ if
\begin{equation} 
  \frac{2}{q} + \frac{d}{p} = \frac{d}{2}, 
\label{L2ad} 
\end{equation}
and $\dot{H}^1$-\textit{admissible} if
\begin{equation} 
  \frac{2}{q} + \frac{d}{p} = \frac{d-2}{2}. 
\label {H1ad} 
\end{equation}
\end{definition}
We define the following Besov norm on $L^2:$ 
\begin{equation*} 
I_k(f) := \left(\int_{2^k \leq |\xi| \leq 2^{k+1}} |\hat{f}(\xi)|^2 d \xi\right)^{1/2},  \hspace{1.5em} \| f \|_B := \sup_{k \in \mathbb{Z}} I_k(f).  
\end{equation*}
The following refinement of the Sobolev inequality is from \cite{Bu}  (Lemma 3.1):
\begin{lemma} (Refined Sobolev)
For $d \geq 3$ there is a constant $C = C(d) > 0$ such that for every 
$u \in \dot{H}^1 (\mathbb{R}^d)$, we have 
\begin{equation} 
  \| u \|_{L^p} \leq C \| \nabla u \|^{\frac{2}{p}}_{L^2} \|\nabla u \|^{1-\frac{2}{p}}_B, 
\label{refined} 
\end{equation} 
where $ p = 2^* = \frac{2d}{d-2}$.
\end{lemma}
The next result, from \cite{G}, provides a decomposition of bounded sequences in 
$L^2(\mathbb{R}^d)$ (for a different, but equivalent Besov norm). 
Here, we specialize to radial functions. 
\begin{proposition}
Let $\{ f_n \}_n $ be a bounded sequence of radially symmetric functions in $L^2(\mathbb{R}^d), d \geq 3$. 
Then there exist a subsequence (still denoted by $\{f_n\}_n$), a sequence of scales
$\{\lambda^j_n\}_n \subset (0,\infty)$ satisfying (\ref{scale orth2}),
and bounded radial $ \{g^j\}_j, \{r_n^J\}_n \subset L^2(\mathbb{R}^d),$ such that for every $J \geq 1$, $x \in \mathbb{R}^d$,
\[
\begin{split}
 &f_n (x) = \sum_{j=1}^J  \frac{1}{(\lambda^j_n)^{d/2}} g^j(\frac{x}{\lambda^j_n})+  r^J_n(x), \\
 &\|f_n\|^2_{L^2} = \sum_{j=1}^J \| g^j \|^2_{L^2} + \| r^J_n \|^2_{L^2} + o_n(1),\\
 &\lim_{J \rightarrow \infty} \limsup_{n \rightarrow \infty} \|r^J_n\|_B = 0.
\end{split}
\]
\end{proposition}
Applying the above result to a bounded radially symmetric sequence $\{v_n\} \subset \dot{H}^1(\mathbb{R}^d),$ we conclude that there is a subsequence (again denoted by $\{v_n\}$), a family of scales $\lambda^j_n$ satisfying~\eqref{scale orth2}), and 
radial $\psi^j \in \dot H^1$ such that for every $J \geq 1$,
\begin{equation} 
  v_n (x) = \sum_{j=1}^J  \psi_n^j(x) +  \tilde{w}^J_n(x)
\label{splitting}
\end{equation} 
where $\psi_n^j(x):=\frac{1}{(\lambda^j_n)^{\frac{d}{2} - 1}} \psi^j(\frac{x}{\lambda^j_n}),$
%$ \{\nabla v^j_n\} = \lambda^j_n-$ oscillatory, $\{\nabla w^J_n\} = \lambda^j_n-$singular, $\forall j \leq J,$ 
\begin{equation*} \|v_n\|^2_{\dot{H}^1} = \sum_{j=1}^J \| \psi^j \|^2_{\dot{H}^1} + \| \tilde{w}^J_n \|^2_{\dot{H}^1} + o_n(1) \end{equation*}
and \begin{equation} \limsup_{n} \| \nabla \tilde{w}^J_n \|_B \xrightarrow{ J \to\infty} 0 \label{error_B} .
\end{equation}
Denote the linear propagator for the homogeneous heat equation 
\begin{equation} v_t = \Delta v \label{heat} \end{equation}
on $\mathbb{R}^d$ by $S(t):=e^{t\Delta}$.
Evolving \eqref{splitting} by the linear propagator we get
\begin{equation*} 
S(t) v_n = \sum_{i=1}^J S(t) \psi^j_n + S(t) \tilde{w}^J_n.
\end{equation*}
Our first goal is to estimate $S(t) \tilde{w}^J_n$ in an appropriate space-time norm. If $\sigma$ is a function on $\mathbb{R}^d,$ we define $\sigma(D)$ by \begin{equation*}\widehat{\sigma(D) f} (\xi) := \sigma(\xi) \hat{f} (\xi).\end{equation*} Also define
$\sigma_k (\xi) := \chi_{2^k \leq |\xi| \leq 2^{k+1}} (\xi) , k \in \mathbb{Z}.$  
Then if $v$ solves \eqref{heat}, by commutation of Fourier multipliers with derivatives, 
$\sigma_k (\xi) v$ is also solves \eqref{heat}.
By standard decay estimate~\eqref{STdecay} for~\eqref{heat}, for any $k$ 
% (since $ \|z(t)\|^2_{\dot{H}^1} = E(z(t)) \leq E(z_0) =  \|z_0\|^2_{\dot{H}^1}$) :
\begin{equation*} 
\|\nabla(\sigma_k (\xi) w(t,\cdot )\|_{L^2} \leq \|\nabla(\sigma_k (\xi) w(0,\cdot)\|_{L^2}. 
\end{equation*}
Using Plancherel's identity and properties of the Fourier transform and the definition of multipliers: 
\begin{equation*} 
\|\nabla (\sigma_k v) \|_{L^2} = \|\widehat{ \nabla(\sigma_k  v)}\|_{L^2} = 
\| |\xi| \sigma_k  \hat{v} \|_{L^2} = I_k(\nabla v).
\end{equation*}
Hence, by taking supremum in $k,$ 
\begin{equation*}
\| \nabla v(t,\cdot) \|_B \leq  \| \nabla v(0,\cdot) \|_B. 
\end{equation*}
This general observation implies  
\begin{equation*} 
\|\nabla (S(t) \tilde{w}^J_n) \|_{L^{\infty}_t B_x} \leq \| \nabla \tilde{w}^J_{n} \|_B. \end{equation*}
Then, due to \eqref{error_B} we conclude
\begin{equation}  \lim_{J \rightarrow \infty} \limsup_{n \rightarrow \infty} \|\nabla (S(t) \tilde{w}^J_n) \|_{L^{\infty}_t B_x} = 0 \label{L00B}. \end{equation}
For every $t > 0, $ \eqref{refined} gives: 
\[
\begin{split}
\| S(t) \tilde{w}^J_n \|_{L^{\frac{2d}{d-2}}_x} &\lesssim \|\nabla(S(t) \tilde{w}^J_n) \|^{\frac{2(d-2)}{2d}}_{L^2} \cdot \|\nabla(S(t) \tilde{w}^J_n) \|^{1-\frac{2(d-2)}{2d}}_B \\
&\lesssim \|\nabla \tilde{w}^J_{n} \|^{\frac{2(d-2)}{2d}}_{L^2} \cdot \|\nabla(S(t) \tilde{w}^J_n) \|^{1-\frac{2(d-2)}{2d}}_B
\end{split}
\]
and so
\[
  \| S(t) \tilde{w}^J_n \|_{L^{\infty}_t L^{\frac{2d}{d-2}}_x} 
  \lesssim \|\nabla \tilde{w}^J_{n}\|_{L^2}^{\frac{2(d-2)}{2d}} \cdot 
  \|\nabla(S(t) \tilde{w}^J_n) \|^{1-\frac{2(d-2)}{2d}}_{L^{\infty}_t B_x}.
\]
Since $\|\nabla \tilde{w}^J_{n}\|_{L^2}$ is uniformly bounded, using \eqref{L00B}: 
\begin{equation}  
\lim_{J \rightarrow \infty} \limsup_{n \rightarrow \infty} \| S(t) \tilde{w}^J_n \|_{L^{\infty}_t L^{\frac{2d}{d-2}}_x} = 0 \label{L00L2*}. 
\end{equation}

It is straightforward to verify, using the Hardy inequality in dimension $d = 2m+2$,
that the map
\[
  X^2 \ni u = u(r) \mapsto v(r) = \frac{u(r)}{r^m} \in \dot H^1_{rad}(\R^d)
\]
is an isomorphism (e.g., see Lemma 4 in \cite{CKM}) and moreover
\[
  u_t = \Delta_m u \;\; \iff \;\; v_t = \Delta_{\R^d} v.
\]
Moreover, if $(r,p)$ is an $L^2$-admissible pair for $d=2$, i.e. 
$\frac{1}{r} + \frac{1}{p} = \frac{1}{2}$, then
\begin{equation*} 
\begin{split}
  \|\frac{u}{r}\|^r_{L^r_t(I; L^p_r(rdr))} &= 
  \int_I \left (\int_0^{\infty}  \frac{r^{1-p}}{r^{2m+2}} |u|^p   r^{2m+1} dr\right)^{r/p} \\
  &= \int_I\left (\int_0^{\infty} \frac{r^{1-p}}{r^{2m+2}} r^{mp} |v|^p   r^{2m+1} dr\right)^{r/p} \\
  & = \int_I \left (\int_0^{\infty} r^{p(m-1)-2m} |v|^p   r^{2m+1} dr\right)^{r/p}.
\end{split}
\end{equation*}
Taking $p = \frac{2m}{m-1}$, and thus $r =2m$, we observe that 
\begin{equation} 
  \| \frac{u}{r} \|_{L^r_t(I; L^p_r(rdr))} =  
  \| v \|_{L^r_t(I; L^p(\mathbb{R}^{2m+2})} ,
\label{link} 
\end{equation} 
and this choice of (r,p) is an $\dot{H}^1$-admissible pair in dimension $2m+2$. 

These observations are the connecting link between the two-dimensional problem and 
the higher-dimensional estimates. 
So for $u_n$ bounded in $X^2$, 
$v_n=\frac{u_n}{r^m}$ is bounded in $\dot{H}^1(\mathbb{R}^d)$ and we have (\ref{splitting}).
First, one has to show that 
\begin{equation}  
  \lim_{J \rightarrow \infty} \limsup_{n \rightarrow \infty} \| S(t) \tilde{w}^J_n \|_{L^{2m}
  (I;L^{\frac{2m}{m-1}}(\mathbb{R}^{2m+2})} = 0 
\label{errorv}. 
\end{equation}
For this we use interpolation and~\eqref{L00L2*} : 
\begin{equation*}
  \|S(t)  \tilde{w}^J_n \|_{L^{2m}_t L^{\frac{2m}{m-1}}} \leq 
  \| S(t)  \tilde{w}^J_n \|^{\frac{m-1}{m}}_{L^{\infty}_t L^{\frac{2m+2}{m}}_r} \cdot 
  \| S(t)   \tilde{w}^J_n \|^{\frac{1}{m}}_{L^2_t L^{\frac{2(m+1)}{m-1}}_r} .
\end{equation*}
Taking $\ds \lim_{J \rightarrow \infty} \limsup_{n \rightarrow \infty},$ and noting that the second term is uniformly bounded 
(by the standard space-time estimates for the heat equation -- see e.g., \cite{Gi}),
the claim \eqref{errorv} follows. 
Undoing the transformation $u_n = r^m v_n$ in (\ref{splitting}) yields
\begin{equation*} u_n(x) = \sum_{j=1}^J \phi^j(\frac{r}{\lambda^j_n}) + w^J_n, \hspace{0.3em} w^J_n = r^m \tilde{w}^J_n. \end{equation*}
   Now, again by interpolation \begin{equation*} \| \frac{u}{r} \|_{L^4_t L^4_r} \leq \| \frac{u}{r} \|^{\frac{m}{2(m-1)}}_{L^{2m}_t L^{\frac{2m}{m-1}}_r} \cdot \| \frac{u}{r}\|^{\frac{m-2}{2(m-1)}}_{L^2_t L^{\infty}_r}. \end{equation*}
Invoking \eqref{errorv}, we get \eqref{holygrail}, i.e., \begin{equation*} \lim_{J \rightarrow \infty} \limsup_n \| e^{t\Delta_m}  w^J_n \|_{L^4_t rL^4_r} = 0. \end{equation*}

The rest of the proof of the profile decomposition follows the same arguments as in the references cited at the beginning of this section, and is thus omitted,
with the exception of the asymptotic energy splitting~\eqref{ED}
which we now demonstrate.
 
Expanding using the definition,
\begin{equation*} \begin{split} 
E(u_n) &= \int_0^{\infty}[\sum_{j=1}^J (\frac{1}{\lambda^j_n})^2(\phi^j_r(\frac{r}{\lambda^j_n}))^2 + \sum_{j=1, i < j}^J \frac{1}{\lambda^j_n \lambda^i_n} \phi^j_r(\frac{r}{\lambda^j_n}) \phi^i_r(\frac{r}{\lambda^i_n}) + (w^J_{n,r}(r))^2 \\ 
& \quad + 2 \sum_{j=1}^J w^J_{n,r}(r) \frac{1}{\lambda^j_n} \phi^j_r(\frac{r}{\lambda^j_n}) + \frac{m^2}{r^2} \sin^2 (\sum_{j=1}^J \phi^j(\frac{r}{\lambda^j_n}) + w^J_{n,r}(r)) ] rdr,
\end{split} \end{equation*} 
We need to show that $E(u_n) - \sum_{j=1}^J E(\phi^j) - E(w^J_n) = o_n(1)$.
Expanding this out, 
\begin{equation*} \begin{split}
& E(u_n) - \sum_{j=1}^J E(\phi^j) - E(w^J_n) \\
& = \int_0^{\infty}\sum_{j=1, i < j}^J \frac{1}{\lambda^j_n \lambda^i_n} \phi^j_r(\frac{r}{\lambda^j_n}) \phi^i_r(\frac{r}{\lambda^i_n}) rdr 
+ 2 \int_0^{\infty} [\sum_{j=1}^J w^J_{n,r}(r) \frac{1}{\lambda^j_n} \phi^j_r(\frac{r}{\lambda^j_n}) ] rdr\\ &+ \int_0^{\infty} \frac{m^2}{r^2}[\sin^2(\sum_{j=1}^J \phi^j(\frac{r}{\lambda^j_n}) + w^J_{n}(r) )- \sum_{j=1}^J \sin^2(\phi^j(\frac{r}{\lambda^j_n})) - \sin^2(w^J_{n}(r))]rdr. 
\end{split} \end{equation*} 

For the first two sums it suffices to look at single pairs and show they all are $o_n(1).$ Using an approximation argument, we can assume every function involved
is in $C^{\infty}_c$ and that all the supports lie in some ball $B(0,R).$ The argument is standard so we only give a sketch: for the first sum, we just change variables, assuming without loss of generality $s^{i,j}_n := \frac{\lambda^j_n}{\lambda^i_n}$ goes to zero. Then, by H\"{o}lder's inequality, each term in the sum is bounded by $\ds \|\phi^i_r\|_{X^2} \int_0^{s^{i,j}_n R} (\phi^j_r(r))^2 rdr = o_n(1).$ 
For the second one, change variables again and employ the weak convergence of $w^J_n(\lambda^j_n r)$ to zero, for all $j \leq J,$ i.e., \eqref{wwconv hmhf}.

For the rest we will use the trigonometric identity 
\begin{equation*}\ds \sin^2(a+b) - \sin^2(a) - \sin^2(b) = \frac{1}{2} \sin(2a)\sin(2b) - 2 \sin^2(a)\sin^2(b)
\end{equation*}
and the inequality derived from it,
\begin{equation} \label{trig ineq} 
  |\sin^2(a+b) - \sin^2(a) - \sin^2(b)| \leq C |a||b|
\end{equation} for some $C>0$.
We want to show that 
\begin{equation*} 
\left|  \int_0^{\infty} \frac{m^2}{r^2}[\sin^2(\sum_{j=1}^J \phi^j(\frac{r}{\lambda^j_n}) + w^J_{n}(r) )
- \sum_{j=1}^J \sin^2(\phi^j(\frac{r}{\lambda^j_n})) - \sin^2(w^J_{n}(r))]rdr   \right| =o_n(1). \end{equation*}
Using (\ref{trig ineq})  $J-1$ times, 
this can be reduced to showing the following two estimates:
\begin{align} 
&\int_0^{\infty}\frac{|\phi^j(\frac{r}{\lambda^j_n})| |\phi^i(\frac{r}{\lambda^i_n})|} {r^2} rdr = o_n(1), \;\; i \neq j \label{prof-prof} \\ 
&\int_0^{\infty} \frac{|w^J_n(r)| \phi^j(\frac{r}{\lambda^j_n})|} {r^2} rdr = o_n(1) \;\; 
\text{for any} \hspace{0.3em} j \leq J. 
\label{error-prof}
\end{align} 

The proof of (\ref{prof-prof}) follows the same rescaling argument as before.

As for (\ref{error-prof}), a change of variables gives 
\[
  \int_0^{\infty} |w^J_n(\lambda^j_n r')| |\phi^j(r')| \frac{r'dr'}{(r')^2}
\]
which suggests that we should use the weak convergence; however,
because of the absolute value, we cannot directly obtain the result.
Since $ \|\frac{\phi^j}{r}\|_{L^2(rdr)} < + \infty,$ for every $\epsilon > 0$ 
we can find an $R = R(\epsilon) > 1$ such that
\begin{equation*} 
  \left (\int_{r \geq R} |\frac{\phi^j(r)}{r}|^2 rdr \right)^{1/2} 
  + \left (\int_{r \leq \frac{1}{R}} |\frac{\phi^j(r)}{r}|^2 rdr \right)^{1/2} < \frac{\epsilon}{2M},\end{equation*}
where 
\begin{equation*}
  M:= \sup_n \|w^J_n\|_{X^2} < +\infty.
\end{equation*}
So if we split the integral at hand into the obvious three regions,
then the inner and outer contributions, by H\"{o}lder's inequality, are 
$< \frac{\epsilon}{2}$, while for the one in the middle we have
\begin{equation*} 
\int_{\frac{1}{R} \leq r \leq R} |w^J_n(\lambda^j_n r)| |\phi^j(r)| \frac{rdr}{r^2} <  \int_{\frac{1}{R} \leq r \leq R} |w^J_n(\lambda^j_n r)| |\phi^j(r)| dr.
\end{equation*}
But, on a fixed interval $\Omega = [a,b], \; 0 < a < b < \infty$, 
$\| u \|_{X^2([a,b])}$ and $\| u \|_{H^1([a,b])}$ are equivalent,
so by the compact embedding of $H^1([a,b])$ in $L^2([a,b])$,
$w^J_n(\lambda^j_n r) \to 0$ in $L^2([a,b])$.
%if it goes weakly to 0 in X^2, take as "test" function, the characteristic of [a,b], which is in the dual, to see that it goes weakly to 0 in $X^2([a,b])$ as well).
%As a matter of fact, more is true: the whole sequence converges strongly by Urysohn's Lemma.\\
%For the whole sequence to converge strongly, by Urysohn's, every subsequence has to have a further subsequence that converges strongly to 0.
%If that is not true, there is one subsequence of the initial one, that has no converging subsequence.\\
%However, this subsequence is bounded too, hence has a weak limit (which by the uniqueness of weak limits is zero) and by the compact embedding it
%has a strongly convergent further subsequence, thus yielding a contradiction to the starting assumption.\\ 
So by H\"older again, we conclude that for $n$ sufficiently large, the integral
in~\eqref{error-prof} is $< \e$, and the result follows.
$\Box$

\subsection{Minimal blow-up solution}

For $u_0 \in E_0$, define 
\begin{equation*} 
  E_c = \inf \{ E(u_0) \; | \;  u \hspace{0.3em} \text{solves} \hspace{0.3em}  (\ref{corotationalCP}) \hspace{0.3em} \text{with} \hspace{0.3em}  u(0)= u_0, \|u\|_{L^4 rL^4([0,T_{\text{max}}))} = + \infty \} ,
\end{equation*}
the infimum of the energies of initial data leading to 
solutions which fail either to be global or to decay to zero, 
in the sense of the local well-posedness theory. 
Note that $T_{\text{max}}$ can be finite (blow-up), or infinite (corresponding to a global but not decaying solution). 

Observe that Theorem \ref{hmhf below threshold theorem} is equivalent to 
$E_c \geq 2 E(Q)$. 
Note also that $E_c > 0$, since for $u_0 \in E_0$, 
$E(u_0)$ small $\; \implies \; \| u_0 \|_{X^2}$ small, and by the local theory, such solutions
are global and decay.

We will follow the contradiction approach of Kenig-Merle: under the assumption
\[
  0 < E_c < 2 E(Q)
\] 
we will first show existence of a {\it critical element} -- a datum with energy $E_c$ giving rise to a solution that that either fails to exist globally or decay to zero. 
Then, as an immediate consequence of energy dissipation, we show that such a critical element cannot exist, reaching a contradiction. 

\begin{proposition} \label{hmhf minimal}
Assume $E_c < 2 E(Q)$. 
There exists $u_{0,c} \in X^2$ with $E(u_{0,c}) = E_c$ such that
if $u_c(t,r)$ is the solution of (\ref{corotationalCP}) with initial data $u_{0,c}$ and maximal interval of existence
$I = [0, T_{\text{max}}(u_{0,c}))$, then $ \|u_c\|_{L^4 rL^4(I)} = + \infty.$
\end{proposition}
For the proof of this proposition we follow the same strategy as in \cite{KMa,KV}. 
\begin{proof}
Let $ \{ u_{0,n} \}_n \subset X^2$ such that 
$E(u_{0,n}) \searrow E_c, n \rightarrow \infty$, and
the corresponding solutions $u_n$ of (\ref{corotationalCP}) with maximal intervals of existence $I_n = [0,T_{\text{max}}(u_{0,n}))$ satisfy
$\| u_n \|_{L^4 rL^4(I_n)} = + \infty.$ By the comparability of the energy and the $X^2-$norm, the sequence $ \{ u_{0,n} \}_n$ is bounded in $X^2$. 
Thus, passing to a subsequence, if necessary, we have the profile decomposition
\begin{equation*} 
  u_{0,n} (r) = \sum_{j=1}^J \phi^j( \frac{r}{\lambda^j_n}) + w^J_n(r)
\end{equation*}
with the stated properties in Proposition \ref{hmhfPD}.

Define the {\it nonlinear profile} $ v^j: I^j \times [0,\infty) \rightarrow \mathbb{R}$ associated to 
$\phi^j$ to be the maximal-lifespan solution to \eqref{corotationalCP} with initial data $\phi^j$,
and for each $j,n \geq 1$, define $v^j_n: I^j_n \times [0,\infty) \rightarrow \mathbb{R}$ by 
\[
  v^j_n(t,r) = v^j(\frac{t}{(\lambda^j_n)^2},\frac{r}{\lambda^j_n}), \qquad
  I^j_n := \{t \in \mathbb{R}^{+} : \frac{t}{(\lambda^j_n)^2} \in I^j \},
\]  
the solution to \eqref{corotationalCP} with initial data 
$v^j_n(0) = \phi^j(\frac{r}{\lambda^j_n})$. 
The energy decoupling reads
\begin{equation*} 
  E(u_{0,n}) = \sum^J_{j=1} E(\phi^j) + E(w^J_n) + o_n(1) \quad \forall J.
\end{equation*}
Taking $\overline{\lim_n}$ , we get
\begin{equation*} 
  E_c = \sum^J_{j=1} E(\phi^j) + \overline{\lim_n} E(w^J_n)
\end{equation*}
which by the positivity of every term implies $\ds \sum^J_{j=1} E(\phi^j) \leq E_c,$ for any $J,$
and so
\begin{equation*}
   \sup_j E(\phi^j) \leq E_c.
\end{equation*}
The goal is to show that $\phi^j =0, j \geq 2$ and $E(\phi^1) = E_c$. 

We consider the following possibilities:\\

\textit{Case 1}: 
$\sup_j E(\phi^j) <  E_c$. 
Then by the definition of $E_c$, each $v^j$ (hence also $v^j_n$) is global
($I^j = [0,\infty)$) and decaying:
$\| v^j \|_{L^4([0,\infty); r L^4)} < \infty$. 
Define an approximate solution (to $u_n(t)$) of the nonlinear equation by
\begin{equation*} 
  u^J_n(t) = \sum^J_{j=1} v^j_n(t) + e^{t \Delta_m} w^J_n.
\end{equation*}
What we want to show is that  $u^J_n$ is a good approximate solution to $u_n$ (for $n,J$ sufficiently large) in the sense of the Stability Theorem \ref{LT perturbations}. This would imply that $u_n(t)$ is global, a contradiction.

First, to see that 
$\ds \sup_{n,J} \| u^J_n \|_{L^4 rL^4(\mathbb{R}^{+})} < +\infty:$ for any $\epsilon > 0,$ \eqref{holygrail} provides a $J$ such that 
\begin{equation*} \begin{split} \overline{\lim_n}  \|u^J_n\|_{L^4 rL^4(\mathbb{R^{+}})} &\leq  \overline{\lim_n}  \| \sum^J_{j=1} v^j_n  \|_{L^4 rL^4(\mathbb{R^{+}})} + \overline{\lim_n}  \|e^{t\Delta_m} w^J_n \|_{L^4 rL^4(\mathbb{R^{+}})} \\ &\leq \sum^J_{j=1} \|v^j \|_{L^4 rL^4(\mathbb{R^{+}})} + \epsilon. \end{split} \end{equation*} 
To conclude the claim, we will show that the latter norms are bounded uniformly in $J$. 
We can split the sum into two parts (for every fixed $J$); one over $ 1 \leq j \leq J_0,$ and the rest. Let $\epsilon_0$ be such that Theorem  \ref{hmhfLWP} guarantees that, if $\|u_0\|_{X^2} \leq \epsilon_0,$ then the corresponding solution $u$ is global
with $\| u \|_{L^4_t r L^4} \leq C \epsilon_0$.
%(scaling with the comparability constant between the energy and the $X^2-$norm), 
Pick $J_0$ such that 
\begin{equation*} \sum_{j \geq J_0} E(\phi^j) \leq \epsilon_0.
\end{equation*}
Then for $j \geq J_0$, the $\| v^j \|_{L^4_t rL^4}$ are uniformly bounded,
and the claim follows.

By construction, we have $\|u^J_n(0) - u_n(0) \|_{X^2} = 0, \;  \forall J,n$, and also, 
$\| e^{t \Delta_m} (u^J_n(0) - u_n(0)) \|_{L^4 rL^4(\mathbb{R^{+}})} = 0, \; \forall J,n$. 

The perturbed PDE for $u^J_n(t)$ is 
\begin{equation*} \partial_t  u^J_n - \Delta_m u^J_n = \sum^J_{j=1} F(v^j_n),
\end{equation*} 
hence the error is given by 
\begin{equation*} e^J_n = F(u^J_n) - \sum^J_{j=1} F(v^j_n),
\end{equation*}
where F is the nonlinear term $F(u) = \frac{m^2}{r^2} ( u - \frac{\sin2u}{2}).$
We will show that the error is small in the dual norm $\| \cdot \|_{L^{4/3}_t rL^{4/3}_r}$
 for sufficiently large $n$ and $J.$
Explicitly,
\begin{equation*} 
  e^J_n = \frac{m^2}{2 r^2} \left( 2u^J_n - \sin(2u^J_n) -  \sum^J_{j=1} (2 v^j_n - \sin(2v^j_n)
  \right).
\end{equation*} 
For simplicity, denote $W^J_n (r,t) := e^{t\Delta_m} w^J_n(r).$ 
We will make use of the following trigonometric relation:
\begin{equation*}
\begin{split} | \sin(2u) + \sin(2v) - \sin(2u+2v)| &= |2 \sin(2u) \sin^2(v) + 2 \sin(2v) \sin^2(u)| \\ 
&\lesssim |u| |v|^2 + |v| |u|^2 .
\end{split} 
\end{equation*}
Using this,
\begin{equation*}
\begin{split} 
|\sum^J_{j=1} \sin(v^j_n) +& \sin(W^J_n) - \sin(\sum^J_{j=1} v^j_n + W^J_n) \pm  \sin(\sum^J_{j=1} v^j_n)| \\&\lesssim |\sum^J_{j=1} \sin(v^j_n) - \sin(\sum^J_{j=1} v^j_n)| +  |W^J_n| |\sum^J_{j=1} v^j_n|^2 + |W^J_n|^2 |\sum^J_{j=1} v^j_n| 
\end{split} 
\end{equation*}
Define $\ds A:= |W^J_n| \hspace{0.3em} |\sum^J_{j=1} v^j_n|^2 + |W^J_n|^2 \hspace{0.3em} |\sum^J_{j=1} v^j_n|,\hspace{0.5em} B:= |\sum^J_{j=1} \sin(v^j_n) - \sin(\sum^J_{j=1} v^j_n)|$.
By H\"{o}lder's inequality:
 \begin{equation*}
 \begin{split}
 \|\frac{1}{r^2} A\|_{L^{4/3} rL^{4/3}} &\leq \|W^J_n\|_{L^{4} rL^{4}} \hspace{0.3em} 
 \|\sum^J_{j=1} v^j_n \|^2_{L^{4} rL^{4}} + \|W^J_n\|^2_{L^{4} rL^{4}} \hspace{0.3em} 
 \|\sum^J_{j=1} v^j_n \|_{L^{4} rL^{4}}\\ &\leq \|W^J_n\|_{L^{4} rL^{4}} (\sum^J_{j=1}\|v^j_n 
 \|_{L^{4} rL^{4}})^2 + \|W^J_n\|^2_{L^{4} rL^{4}} (\sum^J_{j=1}\| v^j_n \|_{L^{4} rL^{4}}).
\end{split}
\end{equation*}
But by \eqref{holygrail}, 
\begin{equation*}
  \lim_{J \rightarrow \infty} \limsup_{n \rightarrow \infty} 
  \|W^J_n\|_{L^{4} rL^{4}} = 0,
\end{equation*}
and hence, by the scaling invariance of the $L^4 rL^4-$norm, 
\begin{equation*}
\lim_{J \rightarrow \infty} \limsup_{n \rightarrow \infty} \|A\|_{L^{4/3} rL^{4/3}} = 0.
\end{equation*}
For term B, again by adding and subtracting 
$\ds \sin(\sum^{J-1}_{j=1} v^j_n)$ we get, using the trigonometric inequality: 
\begin{equation*}
  B \lesssim  |\sum^{J-1}_{j=1} \sin(v^j_n) - \sin(\sum^{J-1}_{j=1} v^j_n)| + 
  |v^J_n| |\sum^{J-1}_{j=1} v^j_n|^2 + |v^J_n|^2 |\sum^{J-1}_{j=1} v^j_n|. 
\end{equation*}
We will show how to treat the second term, and after that the procedure can be easily iterated. It consists of terms of the form $|v^J_n| |v^j_n|^2$ and $|v^J_n|^2 |v^j_n|$.
% (employing Young's inequality for the product terms $|v^j_n||v^i_n|$).
%and triangle inequality for the second
We treat terms of the first type, namely 
$ \| \frac{|v^J_n|^2}{r^2} \hspace{0.3em} \frac{|v^j_n|}{r} \|_{L^{4/3}rL^{4/3}}$ 
(and the others follow in the same way).
We may employ an approximation argument to assume
the functions are smooth and compactly supported in space-time, say on $[0,T] \times [0, R]$.
Without loss of generality, assume $ s^{J,j}_n := \frac{\lambda^J_n}{\lambda^j_n} \rightarrow 0 .$  Changing variables (in space and time), and using H\"{o}lder's inequality, 
we find that the above norm is controlled by
\begin{equation*} 
  \|v^J\|^2_{L^4 rL^4}  \hspace{0.3em} \|v^j\|_{L^4 rL^4 ([0, (s^{J,j}_n)^2 T] \times [0,(s^{J,j}_n) R])} \xrightarrow{ n \to\infty} 0.
\end{equation*}
The other terms may be treated similarly, proving 
\begin{equation*} 
  \lim_{J \rightarrow \infty} \limsup_{n \rightarrow \infty} 
  \|e^J_n\|_{L^{4/3} rL^{4/3}} = 0.
\end{equation*}

Thus, we have shown that for sufficiently large $J$ and $n$, 
$u^J_n(t)$ is a good approximate solution in the sense of 
Stability Theorem \ref{LT perturbations}, from which it follows that 
$u_n(t)$ is global, with $\| u_n \|_{L^4_t r L^4} < \infty$, a contradiction. \\

\textit{Case 2}: $\ds \sup_j E(\phi^j) = E_c$.
This immediately implies (possibly after a relabeling) that $\phi^j = 0$ for $j \geq 2$, 
and the profile decomposition simplifies to 
\begin{equation*} 
  u_{0,n}(r) = \phi^1(\frac{r}{\lambda^1_n}) + w^1_n(r).
\end{equation*}
By the energy splitting and the fact that $E(u_{0,n}) \rightarrow E_c,$
we get 
\begin{equation*} 
  \overline{\lim_n}  E(w^1_n) = 0, 
\end{equation*}
from where, by the comparability of the energy and the $X^2$-norm, 
it follows that 
\[
  \tilde{u}_{0,n}(r) := u_{0,n}(\lambda^1_n r) \rightarrow \phi^1(r)
 \]
strongly in $X^2$. We also get that $E(\phi^1) = E_c.$

Define our critical element $u_c$ to be the solution of~\eqref{corotationalCP} 
emanating from initial data $\phi_1$.
To complete the proof of the Proposition, we must conclude that 
$\| u_c \|_{L^4([0,T_{max}(\phi^1)); rL^4)}= \infty$.
To see that, assume it is false, and again employ the 
Stability Theorem \ref{LT perturbations} as above to reach a contradiction.
%once more with the same approximate
%solution as before (having a single profile now).
%To check \eqref{exp}, notice that by the space-time estimates $\|e^{t\Delta_m} w^1_n\|_{L^4 rL^4} \leq \|w^1_n \|_{X^2} \xrightarrow{n \to\infty} 0$ from \eqref{strongc} and
%the comparability of the energy and the $X^2-$norm. Everything else follows exactly the same arguments as before.
\end{proof} 

\subsection{Rigidity}
In this short section, we complete the proof of Theorem~\ref{hmhf below threshold theorem}
by showing, as an immediate consequence of energy dissipation, that:
\begin{proposition}
The critical element $u_c$ found in Proposition \ref{hmhf minimal} cannot exist.
\end{proposition}
\begin{proof}
By the energy dissipation relation, for $0 < t < T_{max}(\phi_1)$, 
\[
  E(u_c(t)) < E(u_c(0)) = E_c \in (0, 2 E(Q))
\]
unless $u_c$ is a stationary solution.
There are no non-zero stationary solutions in $E_1$, so this strict inequality holds.
Thus $E(u_c(t)) < E_c$ for some $t > 0$, whence it follows from the definition of $E_c$
that $u_c$ is global with $\| u_c \|_{L^4_t r L^4} < \infty$, a contradiction.
\end{proof}

%-------------------------------------------------------------------
\section{Heat Flow Above Threshold} \label{HMHF Above}

In this section we prove Theorem~\ref{above threshold theorem}
on the ``above-threshold" solutions of the corotational heat-flow.

\subsection{Corotational maps in $E_1$} \label{E_1}

Recall, we consider here solutions $u(r,t)$ of
\begin{equation} \label{pde}
  u_t = u_{rr} + \frac{1}{r} u_r + \frac{m^2}{2r^2}\sin(2u)
\end{equation}
in the class
\begin{equation*}
  E_1 := \{ u \; | \;  E(Q) \leq  E(u_0) \leq 3 E(Q), \quad 
  u(0)=\pi, \lim_{r \rightarrow \infty} u (r) = 0 \},
\end{equation*}
where the energy is given by
\[
  E(u) = \frac{1}{2} \int_0^\infty \left( u_r^2 + \frac{m^2}{r^2} \sin^2(u) \right) r dr.
\]
Recall the unique (up to scaling) static solution with these boundary conditions is
\[
  Q(r) = \pi - 2 \tan^{-1}(r^m),
\] 
and define the following quantities
\[
  h(r):= \sin(Q(r)) = \frac{2 r^m}{1 + r^{2m}},  \quad  \hat{h}(r) := \cos(Q(r)) = \frac{r^{2m} - 1}{r^{2m} + 1}.
\]
For later use, we record the easy computations 
\[
  h_r = - \frac{m}{r} h \hat{h}, \qquad \hat{h}_r = \frac{m}{r} h^2.
\]
We will denote scalings by
\[
  Q^s(r) := Q(r/s), \;\; h^s(r) = h(r/s), \;\;  etc., \quad s > 0.
\]
Recall that the energy space (for maps with trivial topology) is:
\[
  X^2 = \{ w : [0,\infty) \mapsto \R \; | \; \int_0^\infty \left( w_r^2 + \frac{w^2}{r^2} \right) r \; dr < \infty \}.
\]  

\subsection{No concentration at spatial infinity}

As discussed in the introduction, the mechanism of possible singularity formation is
well-known: energy concentration by bubbling off static solutions (harmonic maps).
By the corotational symmetry and finite energy, a concentration may a priori 
occur only at the spatial origin or infinity.
The latter cannot happen in finite time:
 \begin{lemma} \label{no concentration} 
 Let $u$ be a finite energy smooth solution on (\ref{pde}) on $(0,T).$ No energy concentration at spatial infinity is possible:
\begin{equation*} \lim_{R \rightarrow \infty} \limsup_{t \rightarrow T} E(u(t) ; B^c_R) = 0. \end{equation*}
\end{lemma}
\begin{proof}
The energy dissipation relation 
\begin{equation} 
  E(u(t_2)) + \int_{t_1}^{t_2} \|u_t\|^2_{L^2} ds  = E(u(t_1))
\label{E-dissipH} 
\end{equation} for $0 \leq t_1 < t_2 < T$ will be used.
First choose a smooth, radial cut-off function $\psi$ such that
\begin{equation*}
\psi(r) = \left\{
\begin{array}{rl}
0 & \text{if } r \leq 1\\
1 & \text{if } r \geq 2
\end{array} \right.
\end{equation*}
and define $\ds \psi_R(r) := \psi(\frac{r}{R})$.

If there was energy concentration at spatial infinity at time $t=T,$
for some $\delta > 0,$ we would have
\begin{equation*} 
  \limsup_{t \nearrow T} E(u(t); B^c_R) \geq \delta > 0,  \;\;  \forall R>0, 
\end{equation*} 
and we could find sequences of radii $R_n \nearrow \infty $ and times $t_n \nearrow T$ such that $\ds \lim_n E(u(t_n), B^c_{R_n}) \geq \delta > 0.$ 
Define the ``exterior'' energy
\begin{equation*} 
  \hat{E}_R(t) := \frac{1}{2} \int_0^{\infty} \psi_R(r) \left(u^2_r + \frac{m^2}{r^2} \sin^2(u) \right) rdr. \end{equation*}
By the finiteness of the energy, for any $t_0 < T$ there is an $R_0 > 1,$ such that $\hat{E}_{R_0}(t_0) \leq \frac{\delta}{4}.$ By assumption, there is  $T > t_1 > t_0$
such that  $\hat{E}_{R_0}(t_1) \geq \frac{\delta}{2}.$

By direct calculation
\[
  \frac{d}{dt}\hat{E}_{R_0}(t) = - \int_0^{\infty}
\psi_{R_0} u^2_t rdr - \int_0^{\infty} u_r u_t \frac{d \psi_{R_0}}{dr} r dr.
\]
Using $\ds \frac{d}{dr} \psi_{R}(r) = \frac{1}{R} \psi'(\frac{r}{R})$
and (~\ref{E-dissipH}):
\begin{equation*} 
\begin{split} 
\frac{\delta}{4} &\leq \int_{t_0}^{t_1} \frac{d}{dt} E(u(t); B^c_{R_0}) dt= -\int_{t_0}^{t_1}  \int_0^{\infty}
\psi_{R_0} u^2_t rdr - \int_{t_0}^{t_1}  \int_0^{\infty} u_r u_t \psi'_{R_0} r dr.  
\\ & \leq \left(\int_{t_0}^{t_1}  \int_0^{\infty} u^2_t rdr \right )^{1/2} \cdot \left(\int_{t_0}^{t_1}  \int_0^{\infty} u^2_r \left(\psi'_{R_0}\right)^2 \right )^{1/2} \\ 
&\lesssim \frac{1}{R_0} (t_1 - t_0)^{1/2} E(u_0),
\end{split} 
\end{equation*}
which yields a contradiction taking $t_0 \nearrow T$.
\end{proof}

%Given this, and following \cite{St} (or \cite{Ber} directly in the corotational setting),
%we have
%\begin{lemma}If the solution blows-up at time $T < \infty$, then
%\begin{equation*} u(t_j,  s(t_j) r) \rightarrow Q(r), \hspace{0.5em} s(t_j) \rightarrow 0\end{equation*}
%in $X^2_{\text{loc}}$ along a sequence $\{t_j\}_n \nearrow T.$
%\end{lemma}
%
%However working locally removes any knowledge of the topology of the map, which is determined by the behavior of the map at spatial infinity. We will improve the above result in the corotational setting by working globally in space in the energy
%topology. Here we are forced to account for the topological restrictions of non-trivial degree maps, and in fact we shall use these restrictions, along with our degree zero
%theory (see the previous section), to our advantage.\\

\subsection{Bubbling description}
 
Having ruled out energy concentration at infinity,
and since the energy bound and boundary conditions in $E_1$ prohibit the formation of more than one bubble,
the following proposition giving the strong convergence of the solution at a 
blow-up time, after removal of the bubble,
is a direct adaptation of Theorem 1.1 in \cite{Q}:
\begin{proposition} 
Let $u_0 \in E_1$ and $u(t)$ the corresponding solution to (\ref{pde}) blowing up at time $t=T>0.$
Then there exists a sequence of times $t_j \nearrow T,$ a sequence of scales $s_j = o(\sqrt{T-t_j}),$ a map $w_0 \in E_0,$ and a decomposition
\begin{equation} u(t_j,r) = Q(\frac{r}{s_j}) + w_0(r) + \xi(t_j,r) \label{b_decomposition} \end{equation}
such that $\ds \xi(t_j,0)=\lim_{r \rightarrow \infty} \xi(t_j,r) =0 $ and $\ds \xi(t_j) \rightarrow 0$ in $X^2$ as $j \rightarrow \infty.$
\end{proposition}

So to exclude finite-time singularity formation, it suffices to show:
\begin{proposition} \label{main in Chapter 3}
Assume $m \geq 4$. 
Suppose $u(t,r)$ is a smooth solution of~\eqref{pde} on $[0,T)$ such that
along some sequence $t_j \to T-$, there are $s_j > 0$ such that
\begin{equation} \label{PropAss}
  u(t_j,\cdot) - Q^{s_j} \to w_0 \mbox{ in } X^2
\end{equation}
for some $w_0 \in X^2$ with $E(w_0) < 2E(Q)$.
Then $s_j \not \to 0$.
\end{proposition}

The next three subsections build up to to a proof of this.

\subsection{Approximate solution}

Introduce the solution $w(t,r)$ of~\eqref{pde} with initial data at $t = t_j$ given by $w_0$:
\begin{equation} \label{w}
  \begin{array}{c}
  w_t - w_{rr}  - \frac{1}{r} w_r - \frac{m^2}{2r^2}\sin(2w) = 0 \\
  w(t_j,r) = w_0(r) \in X^2, \quad E(w_0) < 2 E(Q)
  \end{array}.
\end{equation}
By Theorem~\ref{hmhf below threshold theorem},
we know that $w$ is a global, smooth solution with
\begin{equation} \label{west}
  \| w \|_{L^\infty_t X^2 \cap L^2_t (X^\infty \cap r X^2)([t_j,\infty))} < \infty.
\end{equation}
For later use, we record one consequence of the higher regularity gained after the 
initial time:
\begin{equation} \label{higher}
  \frac{w}{r^2} \in L^2_t X^2 ( [t^*, \infty) ) \; \; \mbox{ for every } t^* > t_j.
\end{equation} 
This follows from the observations that by standard parabolic regularity estimates
(for example by performing energy-type estimates on the differentiated PDE),
the function $v(x,t) = w(r,t) e^{i m \theta}$ satisfies
$D^2 v \in L^2_t L^2_x ( [t^*,\infty) )$ , and that $w/r^2$ and $w_r/r$ are controlled pointwise by $|D^2 v|$ (for any $m \geq 2$).

For fixed $s > 0$, $Q^s$ is also a (static) solution of~\eqref{pde}.
Since the PDE is nonlinear, of course the sum $Q^s + w$ is {\it not} a solution:
\[
\begin{split}
  & \left (\p_t - \p_r^2 - \frac{1}{r} \p_r - \frac{m^2}{2r^2} \sin(2 \; \cdot \;) \right)(Q^s + w) \\
  & \quad = \frac{m^2}{2r^2} \left( \sin(2Q^s) + \sin(2w) - \sin(2Q^s + 2w) \right)  \\
  & \quad = \frac{m^2}{2 r^2} \left( \sin(2Q^s)(1 - \cos(2w)) + \sin(2w)(1 - \cos(2Q^s) \right)
  =: Eqn(Q^s + w).
\end{split}
\] 
However, 
$Q^{s(t)} + w$ is a good {\it approximate solution} over short time intervals in the sense:
\begin{lemma} 
$ \| Eqn(Q^{s(t)} + w) \|_{L^2_t X^1} \lesssim \|  w \|_{L^2_t X^\infty} + \| w \|_{L^4_t X^4}^2$
and therefore by~\eqref{west},
\begin{equation} \label{approxsol}
  \| Eqn(Q^{s(t)} + w) \|_{L^2_t X^1 [t_j,T)} \to 0 \; \mbox{ as } \;
%  \sup_{t \in [t_j,\infty)} s(t) \to 0.
  t_j \to T-.
\end{equation}
\end{lemma}
\noindent {\it Remark}: We do not need it here, but if $0 < s(t) \ll 1$, then $Q^{s(t)} + w$ is a good approximate 
solution {\it globally}, in the sense that
$\ds \|Eqn(Q^{s(t)} + w) \|_{L^2_t X^1 [t_j,\infty)} \to 0$ as $\ds
\sup_{t \in [t_j,\infty)} s(t) \to 0$.
\begin{proof}
This is an easy consequence of the elementary pointwise estimates
\[
\begin{split}
  & |Eqn(Q^{s} + w)| \lesssim \frac{1}{r^2} \left( h^s w^2 + (h^s)^2 w \right) \\
  & |\p_r Eqn(Q^{s} + w)| \lesssim \frac{1}{r^3}  \left( h^s w^2 + (h^s)^2 w \right)
  + \frac{1}{r^2} \left( h^s |w| |w_r| + (h^s)^2 |w_r| \right)  .
\end{split}
\]
Then using $\| \frac{h^s}{r} \|_{L^2} \lesssim 1$, $\| \frac{h^s}{r^2} \|_{L^1} \lesssim 1$,
and H\"older's inequality, the Lemma follows.
\end{proof}

\subsection{Linearized evolution estimates} 

\begin{lemma}
Assume $m \geq 4.$ Let $\xi(\cdot,t) \in X^2$ be a solution of the 
inhomogeneous {\it linearized equation} about $Q^s$, 
where $s = s(t) > 0$ is a differentiable function of time,
\[
  \left\{ \begin{array}{c} \p_t \xi + H^s \xi = f(r,t) \\
  \xi(r,0) = \xi_0(r) \end{array} \right\} \qquad
  H^s := -\p_r^2 - \frac{1}{r} \p_r  - \frac{m^2}{r^2} \cos(2 Q^s),
\]
which also satisfies the orthogonality condition
\begin{equation} \label{orthCh3}
  \left( \xi(\cdot,t), \; h^{s(t)} \right)_{L^2_{rdr}} \equiv 0.
\end{equation}
Then we have the estimates
\begin{equation} \label{linest}
  \| \xi \|_{L^\infty_t X^2 \cap L^2_t X^\infty} \lesssim
  \| \xi_0 \|_{X^2} + \| f \|_{L^1_t X^2 + L^2_t X^1} + \| \dot{s} \|_{L^2_t}.
\end{equation}
\end{lemma}
\begin{proof}
The idea comes from \cite{GNT} where it appeared as a linearization of 
a {\it generalized Hasimoto transformation}, while here we apply it directly at the linear level:
exploit the factorized form of the linearized operator
\[
  H^s = (L^s)^* L^s, \quad L^s =  \p_r + \frac{m}{r} \cos(Q^s) 
  = h^s \p_r (h^s)^{-1},
\] 
and the fact that the reverse factorization is positive,
\begin{equation} \label{posfac}
  (L^s) (L^s)^* = -\p_r^2 - \frac{1}{r} \p_r + \frac{1}{r^2} \left(
  1 + m^2 - 2 m \cos(Q^s) \right) \geq
   -\p_r^2 - \frac{1}{r} \p_r + \frac{(m-1)^2}{r^2}.
\end{equation}
Applying $L^s$ to the linearized equation produces
\[
  \p_t \eta + L^s (L^s)^* \eta = L^s f + (\p_t L^s) \eta
  \qquad \eta := L^s \xi, \;\;  
  \p_t L^s = \frac{m^2}{r} \frac{1}{s} (h^s)^2 \dot  s.
\]
Multiplying this equation by $\eta$, integrating over space and time, and 
using~\eqref{posfac}  gives
\[
  \| \eta \|_{L^\infty_t L^2}^2 + \| \eta \|_{L^2_t X^2}^2 \lesssim
  \| L^s \xi_0 \|_{L^2}^2 + \| (L^s f) \eta \|_{L^1_t L^1} + 
  \| \frac{1}{s} (h^s)^2 \frac{\eta}{r}  \|_{L^2_t L^1} \| \dot s \|_{L^2_t}.
\]
Using H\"older's inequality on the right, then Young's,
as well as \\ $\| \frac{1}{s} (h^s)^2 \|_{L^2} \lesssim 1$, yields
\[
  \| \eta \|_{L^\infty_t L^2 \cap L^2_t X^2} \lesssim
  \| \xi_0 \|_{X^2} + \| f \|_{L^1_t X^2 + L^2_t X^1} 
  + \| \dot s \|_{L^2_t}.
\]
Finally, in \cite{GNT} it was shown that we can invert $L^s$ under the orthogonality condition~\eqref{orthCh3} to bound $\xi$:
\[
  m \geq 4, \;\; \left( \xi, \; h^s \right)_{L^2} = 0 \;\; \implies \;\;
  \| \xi \|_{X^p} \lesssim \| L^s \xi \|_{L^p}, \;\; 2 \leq p \leq \infty. 
\]
Together with the standard embedding $\| \eta \|_{L^\infty} \lesssim \| \eta \|_{X^2}$ 
this completes the proof.
\end{proof}

\subsection{Modulation argument}

\begin{proof} (of Proposition \ref{main in Chapter 3})
%We argue by contradiction, so assume
%\begin{equation} \label{contradict}
%  s_j \to 0 .
%\end{equation}
Let $w(r,t)$ be as in~\eqref{w}.
For $t \in [t_j, T)$, the idea is to write the solution $u(r,t)$ in the form
\begin{equation} \label{form}
  u(r,t) = Q^{s(t)}(r) + w(r,t) + \xi(r,t),
\end{equation}
where $s(t) > 0$ is chosen so that the orthogonality condition~\eqref{orthCh3} holds.
The fact that we can make such a choice follows from a standard
implicit function theorem argument:
\begin{lemma}
There is $\e_0 > 0$ such that for any $s_0 > 0$ and any $\xi \in X^2$ with 
$\| \xi \|_{X^2} \leq \e_0$, there is $0 < s = s(\xi,s_0)$ such that 
\[
  Q^{s_0} + \xi = Q^{s} + \tilde{\xi} \;\; \mbox{ with }
  \left( \tilde{\xi}, \; h^s \right)_{L^2_{r dr}} = 0, \quad
  \left| \frac{s}{s_0} - 1 \right| + \| \tilde{\xi} \|_{X_2} \lesssim \| \xi \|_{X^2} \leq \e_0.
\]
\end{lemma}
\begin{proof}
First take $s_0 = 1$.
For $s > 0$ and $\xi \in X^2$ define
\[
  g(s;\xi) := \left( Q - Q^s + \xi, \; h^s \right)_{L^2_{r dr}},
\]
a smooth function of $s$ and $\xi$ because the spatial decay of $h(r)$
implies $ \| r h(r) \|_{L^2_{rdr}} < \infty $ (provided $m > 2$).
We observe that $g(1;0) = 0$, and
\[
  \p_s g(1;0) = \left(  (-\frac{m}{s} h^s, \; h^s ) + (Q-Q^s + \xi, \; \p_s h^s ) \right)|_{s=1, \xi=0}
  = -m \| h \|_{L^2_{rdr}}^2 \not= 0,
\]
so by the Implicit Function Theorem there is $\e_0 > 0$ such that for all $\xi$ with
$\|\xi\|_{X^2} \leq \e_0$, there is $s = s(\xi)$
with $|s-1| \lesssim \| \xi \|_{X^2}$ such that $g(s;\xi) = 0$.
Then also $\tilde{\xi} := \xi + Q-Q^s$ $\implies$ 
$\| \tilde{\xi} \|_{X^2} \lesssim \| \xi \|_{X^2} + |s-1| \lesssim \| \xi \|_{X^2}$.
The case of general $s_0 > 0$ follows from simple rescaling, and the
scale invariance of the $X^2$-norm.  
\end{proof}
  This lemma shows that as long as
\begin{equation} \label{nbhd}
  \inf_{s > 0} \| u(\cdot,t) - w(\cdot,t) - Q^{s} \|_{X^2} < \e_0,
\end{equation}
we may write $u$ in the form~\eqref{form}, with orthogonality~\eqref{orthCh3} holding.

In particular, \eqref{PropAss} implies that for any $0 < \delta_0 < \e_0$, 
by taking $j$ large enough,
and therefore $\| u(\cdot, t_j) - Q^{s(t_j)} - w_0 \|_{X^2}$ small enough, we may write
\begin{equation} \label{initial}
  u(t_j,r) = Q^{s(0)} + w_0(r) + \xi(r,0), \quad \left( \xi(\cdot,0), \; h^{s(0)} \right) = 0,
  \quad \| \xi(\cdot, 0) \|_{X^2} \leq \delta_0.
\end{equation}
So by continuity, \eqref{nbhd} holds on some non-empty time interval $I = [t_j, \tau)$,
$t_j < \tau \leq T$,
on which we may write $u(r,t)$ as in~\eqref{form} with orthogonality~\eqref{orthCh3}.

Moreover by regularity of $u(r,t)$,
by shrinking $\tau$ even more if needed, we may also assume
\begin{equation} \label{xismall}
  \| \xi \|_{L^\infty_t X^2 \cap L^2_t X^\infty ([t_j, \tau))} \leq \delta_0^{\frac{2}{3}},
\end{equation}
which in particular implies~\eqref{nbhd} for $\delta_0$ sufficiently small.

We will use a standard ``continuity argument". That is, we will carry out all our 
estimates over the time interval $I = [t_j, \tau)$ under the assumption~\eqref{xismall}, 
and then conclude that
we may take $\tau = T$ provided $\delta_0$ is chosen sufficiently small.

Inserting~\eqref{form} into the PDE and using standard trigonometric identities
yields the following equation for $\xi$:
\begin{equation} \label{xieqn}
  (\p_t + H^s) \xi = -m \frac{\dot s}{s} h^s + Eqn(Q^s + w)
  + \frac{m^2}{2r^2} \left( V \sin(2 \xi) + N \right),
\end{equation}
where
\[
  V = \cos(2 Q^s)(\cos(2w)-1)) - \sin(2 Q^s) \sin(2w),
\]
and $N$ contains only terms super-linear in $\xi$ coming from the terms
\begin{equation*}
\begin{split}
  & \cos(2Q^s) [2(w+\xi) - \sin(2(w+\xi))] \hspace{0.3em} \text{and}\\
  & \sin(2Q^s) [1-\cos(2(w+\xi))].
\end{split}
\end{equation*}
 Rather than write out all
the terms of $N$ explicitly, we just record the elementary estimates
\begin{equation} \label{Nest}
\begin{split}
  & | N | \lesssim (h^s + |w|) \xi^2 + |\xi|^3 \\
  & | N_r | \lesssim (1 + |w|)(|w_r| + \frac{h^s}{r}) \xi^2
  + \frac{1}{r} |\xi|^3 + (h^s + |w|) |\xi| |\xi_r| + \xi^2 |\xi_r|.
\end{split}
\end{equation}
Our goal is to estimate all the terms on the right side of~\eqref{xieqn} 
in appropriate space-time norms, so that 
we may apply the linear estimates~\eqref{linest}.

For the first term, using $\| \frac{1}{s} h^s \|_{X^1} \lesssim 1$, we have
\begin{equation} \label{sdotest}
  \| -m \frac{\dot s}{s} h^s \|_{L^2_t X^1} \lesssim \| \dot s \|_{L^2_t}.
\end{equation}

The main estimates for $V$ are
\[
  |V| \lesssim w^2 + h^s |w| \; \implies 
  \| \frac{1}{r^2} V \|_{L^2} \lesssim \| \frac{w}{r} \|_{L^4}^2 + \| \frac{w}{r} \|_{L^\infty},
\]
using $\| \frac{h^s}{r} \|_{L^2} \lesssim 1$, and
\[
\begin{split}
  & |V_r| \lesssim |w||w_r| + \frac{1}{r} (h^s)^2 w^2 + h^s |w_r| + \frac{1}{r} h^s |w| \\
  & \implies  \| \frac{1}{r} V_r \|_{L^2}  \lesssim
  \| \frac{w}{r} \|_{L^4} \| w_r \|_{L^4} + \| \frac{w}{r} \|_{L^4}^2 + \| w_r \|_{L^\infty}
  + \| \frac{w}{r} \|_{L^\infty},
\end{split}
\]
using $\| h^s \|_{L^\infty} \lesssim 1$ and $\| \frac{h^s}{r} \|_{L^2} \lesssim 1$. 
Combining these, we obtain a spatial-norm estimate on the linear term on the right side
of~\eqref{xieqn},
\[
  \| \frac{m^2}{2r^2} V \sin(2 \xi) \|_{X^1} \lesssim
  \left( \| w \|_{X^4}^2 + \| w \|_{X^\infty} \right) \| \xi \|_{X^2},
\] 
and from there a space-time estimate:
\begin{equation} \label{Vest}
  \| \frac{m^2}{2r^2} V \sin(2 \xi) \|_{L^2_t X^1} \lesssim
  \left( \| w \|_{L^4_t X^4}^2 + \| w \|_{L^2_t X^\infty} \right) \| \xi \|_{L^\infty_t X^2},
\end{equation}
where, recall, the time interval over which these norms are taken is $I = [t_j, \tau)$.

Finally, from~\eqref{Nest}, we estimate the nonlinear terms:
\[
  \| \frac{1}{r^3} N \|_{L^1} \lesssim \left( \| \frac{h^s}{r} \|_{L^2} + \| \frac{w}{r} \|_{L^2} \right)
  \| \frac{\xi}{r} \|_{L^4}^2 + \| \frac{\xi}{r} \|_{L^3}^3 \lesssim
  \| \frac{\xi}{r} \|_{L^4}^2 + \| \frac{\xi}{r} \|_{L^3}^3,
\]
and using $\| w \|_{L^\infty} \lesssim \| w \|_{X^2} \lesssim 1$, and $\| h_s \|_{X^2} \lesssim 1$,
\[
  \| \frac{1}{r^2} N_r \|_{L^1} \lesssim
  \| \frac{\xi}{r} \|_{L^4}^2
  + \| \frac{\xi}{r} \|_{L^3}^3 +  \| \frac{\xi}{r} \|_{L^4} \| \xi_r \|_{L^4} + \| \frac{\xi}{r} \|_{L^4}^2 \| \xi_r \|_{L^2}.
\]
These last two give
\[
  \| \frac{1}{r^2} N \|_{X^1} \lesssim \| \xi \|_{X^4}^2 + \| \xi \|_{X^3}^3 + \| \xi \|_{X^4}^2 \| \xi \|_{X^2},
\]
and then the space-time estimate:
\begin{equation} \label{Nest2}
  \| \frac{m^2}{2r^2} N \|_{L^2_t X^1} \lesssim   \| \xi \|_{L^4_t X^4}^2(1 + \| \xi \|_{L^\infty_t X^2}) 
  + \| \xi \|_{L^6_t X^3}^3. 
\end{equation}

Now applying the linear estimates~\eqref{linest} to~\eqref{xieqn}, using 
\eqref{initial}, \eqref{sdotest}, 
\eqref{approxsol} (taking $j$ larger as needed),  \eqref{Vest}, and~\eqref{Nest2}, as well as~\eqref{west}, we get
\begin{equation} \label{estimate1}
\begin{split}
  \| \xi \|_{L^\infty_t X^2 \cap L^2_t X^\infty} &\leq C \left( \delta_0 +
  \| \dot s \|_{L^2_t} + \left( \| w \|_{L^4_t X^4}^2 + \| w \|_{L^2_t X^\infty} \right)
  \| \xi \|_{L^\infty_t X^2} \right. \\ & \left. 
  \qquad \quad +  \| \xi \|_{L^4_t X^4}^2(1 + \| \xi \|_{L^\infty_t X^2}) + \| \xi \|_{L^6_t X^3}^3 \right). 
\end{split}
\end{equation}
By~\eqref{west}, by choosing $j$ larger still, if needed, we can ensure
that on the interval $[t_j, T) \supset I$,
\begin{equation} \label{wsmall}
  C \left( \| w \|_{L^4_t X^4([t_j,T)}^2 + \| w \|_{L^2_t X^\infty ([t_j,T)} \right) < \frac{1}{2},
\end{equation}
so that the estimate~\eqref{estimate1} becomes
\[
   \| \xi \|_{L^\infty_t X^2 \cap L^2_t X^\infty} \lesssim \delta_0 +
  \| \dot s \|_{L^2_t} +  \| \xi \|_{L^\infty_t X^2 \cap L^2_t X^\infty}^2 +
   \| \xi \|_{L^\infty_t X^2 \cap L^2_t X^\infty}^3,
\]
and then by using~\eqref{xismall}, 
\begin{equation} \label{estimate2}
   \| \xi \|_{L^\infty_t X^2 \cap L^2_t X^\infty} \lesssim \delta_0 +
  \| \dot s \|_{L^2_t}.
\end{equation}

It remains to estimate $\dot s$. For this, we differentiate the orthogonality 
relation~\eqref{orthCh3}, rewritten as
\[
  \left( \xi,  \frac{1}{s} h^s \right)_{L^2_{r dr}}
\]
for convenience of calculation, with respect to $t$, and use the equation~\eqref{xieqn} 
for $\xi$:
\[
\begin{split}
  0 =& \left( \xi, \; -\dot s  \frac{1}{r^2} (r(r^2 h)')^s \right)\\
  +& \left(  -m \frac{\dot s}{s} h^s + Eqn(Q^s + w)
  + \frac{m^2}{2r^2} \left( V \sin(2 \xi) + N \right), \; \frac{1}{s} h^s \right)
\end{split}
\] 
where we used $H^s h^s = 0$. The first term is bounded by
\[
  |\dot s | \| \frac{\xi}{r} \|_{L^2} \| \frac{1}{r} (r(r^2 h)')^s \|_{L^2}
  \lesssim  |\dot s | \| \xi \|_{X^2},
\]
while
\[
   \left(  -m \frac{\dot s}{s} h^s, \frac{1}{s} h^s \right) = - m \dot s 
   \| \frac{1}{r} (r h)^s \|_{L^2}^2 = - m \dot s \| h \|_{L^2}^2,
\]
so
\begin{equation} \label{seqn}
  \left( m \| h \|_{L^2}^2  + O( \| \xi \|_{X^2}) \right) \dot s 
  =  \left(  Eqn(Q^s + w)
  + \frac{m^2}{2r^2} \left( V \sin(2 \xi) + N \right), \; \frac{1}{s} h^s \right).
\end{equation}
Then by~\eqref{xismall} and $\| r \frac{1}{s} h^s \|_{L^\infty} = \| r h \|_{L^\infty} \lesssim 1$,
\[
  | \dot s | \lesssim 
  \| Eqn(Q^s + w) \|_{X^1}  + \| \frac{1}{r^2} V \sin(2\xi) \|_{X^1} 
  + \| \frac{1}{r^2} N \|_{X^1},
\]
and so by~\eqref{approxsol}, \eqref{Vest}, \eqref{west}, \eqref{Nest} and~\eqref{xismall}:
\[
  \| \dot s \|_{L^2_t} \lesssim 
  \delta_0 + \left( \| w \|_{L^4_t X^4}^2 + \| w \|_{L^2_t X^\infty} \right)
  \| \xi \|_{L^\infty_t X^2} .
\]
Using~\eqref{estimate2} then shows
\[
  \| \dot s \|_{L^2_t} \leq C \left( 
  \delta_0 + \left( \| w \|_{L^4_t X^4}^2 + \| w \|_{L^2_t X^\infty} \right)
  \| \dot s \|_{L^2_t} \right).
\]
As above, by taking $j$ larger if needed we can ensure~\eqref{wsmall} and so
(using again~\eqref{xismall})
\begin{equation} \label{estimate3}
  \| \dot s \|_{L^2_t} + \| \xi \|_{L^\infty_t X^2 \cap L^2_t \cap X^\infty} \lesssim \delta_0.
\end{equation}
This now shows that in our bootstrap assumption~\eqref{xismall}, since we take
$\delta_0 \ll \delta_0^{2/3}$, we may indeed take $\tau = T$, and all of our
previous estimates hold on the full time interval $[t_j,T)$. 

It remains to show that $s(t)$ stays bounded away from zero.
Recall the pointwise bounds used above
\[
\begin{split}
  &|V| \lesssim w^2 + h^s |w|,\\
  &|N| \lesssim (h^s + |w|) \xi^2 + |\xi|^3,\\
  &Eqn(Q^s + w) \lesssim \frac{1}{r^2} \left( h^s w^2 + (h^s)^2 w \right).
\end{split}
\]
We isolate the term in the equation~\eqref{seqn} for $s$ 
coming from the part of $Eqn(Q^s + w)$ which behaves linearly in $w$, 
and write:
\[
  \frac{\dot s}{s} = v_1 + v_2 + v_3
\]
where
\[
  |v_1| \lesssim  \| \frac{w}{r^3} \|_{L^2}  \| \frac{1}{s} (r h^3)^s \|_{L^2}
  \lesssim \| \frac{|w|}{r^2} \|_{X^2} \in L^2_t,
\]
\[ 
  |v_2| \lesssim \| \frac{w^2}{r^2} \|_{L^\infty} \| \frac{1}{s^2} (h^2)^s \|_{L^1} 
  \lesssim \| w \|_{X^\infty}^2 \in L^1_t
\]  
and  
\[
\begin{split}
  |v_3| & \lesssim   
 \| \frac{w^2}{r^2}  \|_{L^2} \| \frac{\xi}{r} \|_{L^\infty}  \| \frac{1}{s} (rh)^s \|_{L^2} +
 \| \frac{w}{r} \|_{L^\infty} \| \frac{\xi}{r} \|_{L^\infty} \| \frac{1}{s^2} h^s \|_{L^1} \\ & \quad +
 \| \frac{\xi^2}{r^2} \|_{L^\infty} \| \frac{1}{s^2} (h^s)^3 \|_{L^1} +
 \| \frac{\xi^2}{r^2} \|_{L^2} (\| \frac{w}{r} \|_{L^\infty} + \| \frac{\xi}{r} \|_{L^\infty} )
 \| \frac{1}{s} (rh)^s \|_{L^2} \\
 & \lesssim \| w \|_{X^4}^2 \| \xi \|_{X^\infty} + \| w \|_{X^\infty} \| \xi \|_{X^\infty}
 + \| \xi \|_{X^\infty}^2 + \| \xi \|_{X^4}^2( \| w \|_{X^\infty} + \| \xi \|_{X^\infty} ) \\
 & \quad \in L^1_t.
\end{split}
\]
So $\frac{\dot s}{s} \in L^1_t + L^2_t$ over $[t^*,T)$, and by the Fundamental 
Theorem of Calculus, and Cauchy-Schwartz,
\[
  \sup_{t^* \leq t < T} \left| \log\left(\frac{s(t)}{s(t^*)}\right) \right| \leq \| \frac{\dot s}{s} \|_{L^1_t([t^*,T))}
  + \sqrt{t-t^*} \| \frac{\dot s}{s} \|_{L^2_t([t^*,T))} < \infty,
\]
so that $s(t)$ remains bounded away from zero, as required.
\end{proof}

\subsection{Completion of the proof}

Proposition \ref{PropAss} shows that a solution of~\eqref{pde}, with 
$m \geq 4$ and $u_0 \in E_1$ cannot form a finite-time singularity. 
Hence such a solution is global.
Moreover, it cannot form a singularity at infinite time $t=\infty$, since such this
would produce a sequence $t_j \to \infty$, with $0 < s_j \to 0 \mbox{ or } \infty$, along which 
$u(\cdot,t_j) - Q^{s_j} \to v_0$ with $X^2 \ni v_0$ a {\it static solution}, hence
$v_0 \equiv 0$. This is however prohibited by the asymptotic stability result of \cite{GNT}.  
Hence we must have $E ( u(\cdot, t) - Q^{s_\infty} ) \to 0$ for some $s_\infty > 0$.
Uniform convergence then follows from the embedding $X^2 \subset L^\infty$.
This completes the proof of Theorem~\ref{above threshold theorem}.
$\Box$

\section*{Acknowledgements}
The first author's research is supported through NSERC Discovery Grant 22124-12. 
The second author is supported by the European Research Council (grant no.~637995 ``ProbDynDispEq''). The results reported in this article were obtained as part of the second author's PhD thesis at the University of British Columbia.
%%%%%%%%%%%%%%%%%%%%%%%%%%%%%
%%%%%%%%%%%%%%%%%%%%%%%%%%%%%
%%%%%%%%%%%%%%%%%%%%%%%%%%%%%

\bigskip

\centerline{\scshape Stephen Gustafson}
\smallskip
{\footnotesize
  \centerline{Department of Mathematics, University of British Columbia}
  \centerline{Vancouver, V6T 1Z2, Canada}
  \centerline{\email{gustaf@math.ubc.ca}}
}

\bigskip

\centerline{\scshape Dimitrios Roxanas}
\smallskip
{\footnotesize
 \centerline{Department of Mathematics, The University of Edinburgh}
\centerline{James Clerk Maxwell Building, The King's Buildings, Peter Guthrie Tait Road}
\centerline{Edinburgh, EH9 3FD,United Kingdom}
\centerline{\email{dimitrios.roxanas@ed.ac.uk}}
}

\end{document}